\documentclass{amsart}

\usepackage{amsmath}
\usepackage{amsthm}
\usepackage{amsfonts}
\usepackage{amssymb}
\usepackage{enumerate}

\DeclareMathOperator{\tr}{tr}

\DeclareMathOperator{\dist}{dist}

\newcommand{\eps}{\varepsilon}
\newcommand{\Prob}{\mathbb{P}}
\newcommand{\C}{\mathbb{C}}
\newcommand{\E}{\mathbb{E}}
\newcommand{\R}{\mathbb{R}}

\renewcommand{\P}{\mathbb{P}}

\renewcommand{\Re}{\mathop{\rm Re}}
\renewcommand{\Im}{\mathop{\rm Im}}

\theoremstyle{plain}
  \newtheorem{theorem}{Theorem}

  \newtheorem{proposition}[theorem]{Proposition}
  
  \newtheorem{lemma}[theorem]{Lemma}
  \newtheorem{corollary}[theorem]{Corollary}

\theoremstyle{definition}
  \newtheorem{definition}[theorem]{Definition}
  
  \newtheorem{remark}[theorem]{Remark}

\begin{document}

\title[Random matrices with external source]{Universality of local eigenvalue statistics in random matrices with external source}

\author[S. O'Rourke]{Sean O'Rourke}
\thanks{S. O'Rourke is supported by grant AFOSAR-FA-9550-12-1-0083.}
\address{Department of Mathematics, Yale University, New Haven , CT 06520, USA  }
\email{sean.orourke@yale.edu}

\author{Van Vu}
\thanks{V. Vu is supported by research grants DMS-0901216, DMS-1307797, and AFOSAR-FA-9550-12-1-0083.}
\address{Department of Mathematics, Yale University, New Haven , CT 06520, USA}
\email{van.vu@yale.edu}

\begin{abstract}
Consider a random matrix of the form $W_n := \frac{1}{\sqrt{n}} M_n + D_n$, where $M_n$ is a Wigner matrix and $D_n$ is a real deterministic diagonal matrix ($D_n$ is commonly referred to as an external source in the mathematical physics literature). 
  We study the universality of the local eigenvalue statistics of $W_n$ for a general class of Wigner matrices $M_n$ and diagonal matrices $D_n$.

  Unlike the setting of many recent results concerning universality,  the global semicircle law fails for this model.
However,  we can still obtain the universal sine kernel formula for the correlation functions. This demonstrates 
the remarkable phenomenon that local laws are more resilient than global ones. 

The universality of the correlation functions follows from 
a four moment theorem, which we prove using a variant of 
the approach used earlier by Tao and Vu.

\end{abstract}

\maketitle

\section{Introduction}
We begin with the class of Hermitian random matrices with independent entries originally introduced by Wigner \cite{W}.  

\subsection{Wigner Random Matrices}

\begin{definition}
Let $n$ be a large number.  A \emph{Wigner matrix} of size $n$ is defined as a random Hermitian $n \times n$ matrix $M_n = (\zeta_{ij})_{i,j=1}^n$ where
\begin{itemize}
\item $\{\zeta_{ij} : 1 \leq i \leq j \leq n\}$ is a collection of independent random variables.
\item For $1 \leq i < j \leq n$, $\zeta_{ij}$ has mean zero and unit variance.  
\item For $1 \leq i \leq n$, $\zeta_{ii}$ has mean zero and variance $\sigma^2$.  
\end{itemize}
\end{definition}

The prototypical example of a Wigner real symmetric matrix is the Gaussian orthogonal ensemble (GOE).  The GOE is defined by the probability distribution 
\begin{equation} \label{eq:distGOE}
	\Prob(d H) = Z^{(\beta)}_n \exp\left({-\frac{\beta}{4}\tr{H^2}}\right) d H
\end{equation}
on the space of $n \times n$ real symmetric matrices when $\beta = 1$ and $d H$ refers to the Lebesgue measure on the $n(n+1)/2$ different elements of the matrix.  Here $Z_n^{(\beta)}$ denotes the normalization constant.  So for a matrix $H=\left( h_{ij} \right)_{i,j=1}^n$ drawn from the GOE, the elements $\{ h_{ij} : 1 \leq i \leq j \leq n \}$ are independent Gaussian random variables with mean zero and variance $1+\delta_{ij}$.  

The classical example of a Wigner Hermitian matrix is the Gaussian unitary ensemble (GUE).  The GUE is defined by the probability distribution given in \eqref{eq:distGOE} with $\beta = 2$, but on the space of $n \times n$ Hermitian matrices.  Thus, for a matrix $H=\left( h_{ij} \right)_{i,j=1}^n$ drawn from the GUE, the $n^2$ different elements of the matrix, 
\begin{equation*}
	\{\Re h_{ij} : 1 \leq i \leq j \leq n \} \cup \{\Im  h_{ij} : 1 \leq i < j \leq n \}
\end{equation*}
are independent Gaussian random variables with zero mean and variance $(1+\delta_{ij})/2$.

We introduce the following notation.  Given a $n \times n$ Hermitian matrix $A$, we let 
$$ \lambda_1(A) \leq \cdots \leq \lambda_n(A)$$
denote the ordered eigenvalues of $A$.  Let $u_1(A), \ldots, u_n(A)$ be the corresponding orthonormal eigenvectors.  We also define the empirical spectral measure $\mu_A$ of $A$ as
$$ \mu_{A} := \frac{1}{n} \sum_{i=1}^n \delta_{\lambda_i(A)}. $$
Corresponding to the empirical spectral measure is the empirical spectral distribution (ESD) of $A$ given by
$$ F^A(x) := \frac{1}{n} \# \left\{ 1 \leq i \leq n : \lambda_i(A) \leq x \right \}, $$
where $\#E$ denotes the cardinality of the set $E$.  

A classic result for Wigner random matrices is Wigner's semicircle law \cite[Theorem 2.5]{BSbook}.  We define the semicircle distribution $F_{\mathrm{sc}}$ by its density 
$$ F'_{\mathrm{sc}}(x) := \left\{
     \begin{array}{lr}
        \frac{1}{2 \pi} \sqrt{4-x^2}, &  \text{if } |x| \leq 2,\\
       0, &  \text{otherwise.}
     \end{array}
   \right. $$
   
\begin{theorem}[Wigner's Semicircle law] \label{thm:semicirclelaw}
For each $n \geq 1$, let $M_n$ be a Wigner matrix of size $n$, where the entries above the diagonal are iid copies of a random variable $\zeta$ and whose diagonal entries are iid copies of $\tilde{\zeta}$.    Then the ESD of $\frac{1}{\sqrt{n}}M_n$ converges to the semicircle distribution $F_{\mathrm{sc}}$ almost surely as $n \rightarrow \infty$.  
\end{theorem}
\begin{remark}
Wigner's semicircle law holds in the case when the entries of $M_n$ are not identically distributed (but still independent) provided the entries satisfy a Lindeberg-type condition.  See \cite[Theorem 2.9]{BSbook} for further details.  
\end{remark}

Wigner's semicircle law describes the global behavior of the eigenvalues.  In particular, if $\{M_n\}_{n\geq 1}$ is the sequence of Wigner matrices from Theorem \ref{thm:semicirclelaw}, it follows that, for any fixed interval $I \subset \mathbb{R}$,
\begin{equation*} 
	\frac{1}{n} \# \left\{ 1 \leq i \leq n : \lambda_i \left( \frac{1}{\sqrt{n}} M_n \right) \in I \right\} \longrightarrow \int_{I} F_{\mathrm{sc}}'(x) dx 
\end{equation*}
almost surely as $n \rightarrow \infty$.  

A major problem in random matrix theory is to study 
the distribution of the eigenvalues at a local scale, and in  the last ten years or so  there have been a series of 
important developments
(see   \cite{E,EKY,ESY1,ESY2,ESY3,ESY4,ESYY,EY,EYY1,EYY2,EYY,TVuniv,TVedge, TViid, TVmeh,  TVsur} and references therein).

  An important ingredient in the  study of the local eigenvalue statistics of Wigner matrices are the correlation functions. The point process $\{\lambda_1,\ldots,\lambda_n\}$ of eigenvalues of a random $n \times n$ Hermitian matrix 
can be described using the \emph{$k$-point correlation functions} $\rho^{(k)} = \rho^{(k)}_n: \R^k \to \R$, defined for any fixed natural number $k$ by requiring that

\begin{equation}\label{cord}
 \E \sum_{\substack{ i_1,\ldots,i_k\\ \text{distinct} }} \varphi (\lambda_{i_1},\ldots,\lambda_{i_k}) = \int_{\R^k} \varphi (x_1,\ldots,x_k) \rho^{(k)}(x_1,\ldots,x_k)\ dx_1\ldots dx_k
 \end{equation}
for any continuous, compactly supported test function $\varphi: \mathbb{R}^k \to \mathbb{C}$.  We refer the reader to \cite{HKPV} for further details regarding random point processes and correlation functions.  

\begin{remark}  When  $x_1,\ldots,x_k$ are distinct and the entries of the random matrix have absolutely continuous joint distribution,  one can interpret $\rho^{(k)}(x_1,\ldots,x_k)$ as the unique quantity such that the probability that there is an eigenvalue in each of the intervals  $(x_i -\eps, x_i +\eps)$ for $i=1,\ldots,k$ is $(\rho^{(k)}(x_1,\ldots,x_k)+o_{\eps \to 0}(1)) (2\eps)^k$ in the limit $\eps \to 0$.
\end{remark}

One of the main tools in random matrix theory is  the Stieltjes transform (of the empirical spectral measure of $\frac{1}{\sqrt{n}} M_n$).  We remind the reader that if $\mu$ is a probability measure on the real line, its Stieltjes transform is defined by
\begin{equation} \label{eq:def:stmu}
	m_{\mu}(z) := \int_{\R} \frac{d\mu(x)}{x - z} 
\end{equation}
for $z \in \C^+ := \{ w \in \mathbb{C} : \Im(w) > 0 \}$.  For a $n \times n$ Hermitian matrix $A$, we introduce the notation
$$ m_{A}(z) := m_{\mu_A}(z) = \int_{\mathbb{R}} \frac{ d \mu_{A}(x)}{x - z }, \quad z \in \C^+. $$
We let $m_{\mathrm{sc}}$ denote the Stieltjes transform of the semicircle distribution.  That is, 
$$ m_{\mathrm{sc}}(z) := \int_{\mathbb{R}} \frac{dF_{\mathrm{sc}}(x)}{x - z} = \frac{1}{2\pi} \int_{-2}^2 \frac{\sqrt{4-x^2}}{x - z} dx, \quad z \in \mathbb{C}^+. $$
It is well known that $m_{\mathrm{sc}}$ is the unique solution of 
\begin{equation} \label{eq:scst}
	m_{\mathrm{sc}}(z) + \frac{1}{z + m_{\mathrm{sc}}(z)} = 0, \quad z \in \C^+ 
\end{equation}
that satisfies $m_{\mathrm{sc}}(z) \rightarrow 0$ as $|z| \rightarrow \infty$.  

From standard Stieltjes transform techniques (see for example \cite[Theorem B.9]{BSbook}), it follows that the conclusion of Theorem \ref{thm:semicirclelaw} is equivalent to 
\begin{equation} \label{eq:sttconv}
	m_{\frac{1}{\sqrt{n}} M_n}(z) \longrightarrow m_{\mathrm{sc}}(z) 
\end{equation}
almost surely as $n \rightarrow \infty$ for each fixed $z \in \C^+$.  

A key observation \cite{BSbook} that has been used by many researchers is that information about the eigenvalues in a short interval centered at $\Re (z)$
 can be established by investigating  \eqref{eq:sttconv} when $\Im(z)$ is allowed to approach zero as $n$ tends to infinity.  

Correlation functions of so-called generalized Wigner random matrices have also been studied (see \cite{EY} for a survey).  In these models, the variances of the entries are not required to be the same, but the variances must satisfy a number of constraints to ensure the semicircle law still holds.

\subsection{Random matrices with external source}

In this note, we study deformed Wigner random matrices.  That is, matrices of the form
\begin{equation} \label{eq:Wn}
	W_n := \frac{1}{\sqrt{n}} M_n + D_n, 
\end{equation}
where $M_n$ is a Wigner matrix of size $n$ and $D_n$ is a deterministic real diagonal matrix.  
The matrix $D_n$ is usually referred to as an external source in the mathematical physics literature
(see \cite{BKext2,BKext, BKext3, CW} and references therein). 

As in the Wigner case above, there exists a limiting ESD that describes the global behavior of the eigenvalues of deformed Wigner random matrices.  In this case, the limiting distribution depends on the sequence of matrices $\{D_n\}_{n \geq 1}$; in many cases the limiting distribution is no longer the semicircle law.  

Assume the spectral norm $\|D_n\|$ of $D_n$ is uniformly bounded in $n$; suppose there exist a probability measure $\mu_D$ such that $\mu_{D_n} \longrightarrow \mu_D$ as $n \rightarrow \infty$.  Then, for a large class of Wigner matrices, there exists a probability measure $\mu_W$ such that $\mu_{W_n} \longrightarrow \mu_W$ almost surely as $n \rightarrow \infty$ \cite{AGZ,K,P,V}.  In particular, $\mu_W$ is characterized by Pastur's relation
\begin{equation} \label{eq:prst}
	m_{\mu_W}(z) = \int \frac{ d \mu_D(x)}{x - z - m_{\mu_W}(z)}, \quad z \in \C^+. 
\end{equation}
Clearly when $\mu_D = \delta_0$, \eqref{eq:prst} reduces to the semicircle case given in \eqref{eq:scst}.  It follows from the results of Biane \cite{B} that $\mu_W$ has continuous density $\rho$ characterized by the Stieltjes inversion formula
\begin{equation} \label{eq:def:rho}
	\rho(x) = \lim_{y \searrow 0} \frac{1}{\pi} \Im \left(m_{\mu_W}(x + \sqrt{-1} y) \right). 
\end{equation}

In this paper, we study a local version of the above limiting law under a number of assumptions on $M_n$, $\{D_n\}_{n \geq 1}$, and $\rho$.  Before stating  our main results, let us  describe some  previous results on the local behavior of eigenvalues of Wigner random matrices with external source.  

In \cite{BKext}, Bleher and Kuijlaars consider the case when the matrix $M_n$ is drawn from the GUE.  In particular, they consider the random matrix ensemble
\begin{equation} \label{eq:nunw}
	\nu_n(dW_n) =  \frac{1}{Z_n} \exp\left({- n \tr \left(\frac{1}{2} W_n^2 -D_n W_n \right)}\right) dW_n
\end{equation}
defined on $n \times n$ Hermitian matrices, where the diagonal matrix $D_n$ has two eigenvalues $\pm a$ each with equal multiplicity.  In this setting the eigenvalues of $W_n$ form a determinantal random point process with kernel $K_n(x,y)$, which can be expressed in terms of a solution to a Riemann-Hilbert problem \cite{BKorth}; see \cite{HKPV,Sdet} for further details regarding determinantal random point processes.  

\begin{theorem}[Bleher-Kuijlaars, \cite{BKext}] \label{thm:BKgauss}
Let $W_n$ be drawn from the set of $n \times n$ Hermitian matrices by \eqref{eq:nunw}, where the diagonal matrix $D_n$ has two eigenvalues $\pm a$ each with equal multiplicity and assume $a > 1$.  Then the mean limiting density of eigenvalues
$$ \rho(x) = \lim_{n \rightarrow \infty} \frac{1}{n} K_n(x,x) $$
exists, and is supported on two intervals $[-\alpha, - \beta]$ and $[\beta, \alpha]$.  The density $\rho$ is real analytic on $(-\alpha, -\beta) \cup (\beta, \alpha)$ and can be expressed as 
$$ \rho(x) =  \frac{1}{\pi} \Im( \xi_{1+}(x)), $$
where $\xi_1$ is the inverse of
$$ z = \frac{ \xi^3 - (a^2 - 1)\xi}{\xi^2 - a^2} $$
with $\xi_1(z) \sim z$ as $z \rightarrow \infty$ and $\xi_{1+}$ is the analytic continuation of $\xi_1$ to $\mathbb{C} \setminus ([-\alpha,-\beta]\cup [\beta,\alpha])$.  Moreover, for every $x_0 \in (-\alpha,-\beta)\cup(\beta,\alpha)$ and $u,v \in \mathbb{R}$, we have
\begin{equation} \label{eq:convsine}
	\lim_{n \rightarrow \infty} \frac{1}{n \rho(x_0)} \hat{K}_n \left( x_0 + \frac{u}{n \rho(x_0)}, x_0 + \frac{v}{n \rho(x_0)} \right) = \frac{ \sin \pi (u-v)}{\pi(u-v)}, 
\end{equation}
where $\hat{K}_n(x,y) = e^{n(h(x) - h(y))} K_n(x,y)$ and
$$ h(x) = - \frac{1}{4} x^2 + \Re (\eta_{1+}(x)), \quad x \in (-\alpha, -\beta) \cup (\beta, \alpha) $$
with $\eta_{1+}(x) = \int_{\alpha}^x \xi_{1+}(s) ds$.  
\end{theorem}

\begin{remark} \label{rem:uniform}
The convergence in \eqref{eq:convsine} actually holds uniformly on compact subsets \cite{AKper}.  This can be seen from \cite[equation (8.6)]{BKext}, which holds uniformly in $z$.  One can then check that the remaining estimates in \cite[Section 8]{BKext} as well as \cite[equations (9.1), (9.2), (9.9), (9.12)]{BKext} hold uniformly.  
\end{remark}

A remarkable  point here is the appearance of the sine kernel, which shows that locally the eigenvalues interact in the same way as in the GUE and other Wigner models.  This leads one to conjecture
 that local laws are more resilient than the global law 
(notice that the semicircle law has completely disappeared from the picture).  This serves as the main motivation of our work.

Results similar to Theorem \ref{thm:BKgauss}  were also obtained by Shcherbina in \cite{Sbulk,Sedge} for both the bulk and edge regimes for a large class of diagonal matrices $D_n$.  In \cite{CW}, Claeys and Wang study an ensemble of random matrices that generalizes \eqref{eq:nunw} when $D_n$ has equispaced eigenvalues.

Following the spirit of recent developments concerning universality, we are going to extend Theorem \ref{thm:BKgauss} for a large class of random matrices in which the entries of $M_n$ are no longer Gaussian. As a matter of fact, we are going to prove 
a stronger theorem, which can be seen as the analogue of the Four Moment theorem from \cite{TVuniv}. The universaliy of the correlation functions follows immediately by the arguments from \cite{TVuniv}. 

 In \cite{LS}, Lee and Schnelli study random matrices of the form
$$ W_n = \frac{1}{\sqrt{n}} M_n + V_n, $$
where $M_n$ is a Wigner matrix whose entries satisfy a sub-exponential decay condition, and $V_n$ is diagonal matrix, independent of $M_n$, whose diagonal entries are iid copies of random variable $v$.  Under some additional assumptions on $v$ and the limiting density $\rho$, the authors prove a local limit law for the Stieltjes transform $m_{W_n}$.  The local limit law is then used to prove a delocalization result for the eigenvectors and a rigidity result for the eigenvalues.
These are important steps in proving universality.  However, the assumptions on $V_n$ in \cite{LS} and our assumption on 
$D_n$  are completely different.  However, the similarities between the results in \cite{LS} and those below hint that a more general treatment of the subject may be possible.

We use asymptotic notation, such as $O,o,\Theta$, under the assumption that $n \rightarrow \infty$.  See Section \ref{sec:notation} for a complete description of the asymptotic notation used here and throughout the paper.

\section{Main Results} \label{sec:results}

We consider a general class of Wigner random matrices whose entries satisfy a sub-exponential decay condition.  

\begin{definition}[Condition {\bf C0}] \label{def:C0}
A random Hermitian matrix $M_n = (\zeta_{ij})_{1 \leq i,j \leq n}$ of size $n$ is said to obey condition {\bf C0} if
\begin{itemize}
\item The $\zeta_{ij}$ are independent (but not necessarily identically distributed) for $1 \leq i \leq j \leq n$.  For $1 \leq i < j \leq n$, they have mean zero and unit variance; for $i=j$, they have mean zero and variance $\sigma^2$ for some fixed $\sigma^2>0$ independent of $n$.  
\item There exist constants $C,C'>0$ such that
$$ \P(|\zeta_{ij}| \geq t^C) \leq \exp(-t) $$
for all $t \geq C'$ and every $1 \leq i \leq j \leq n$.  
\end{itemize}
\end{definition}

A large class of Hermitian and real symmetric Wigner matrices obey condition {\bf C0}.  For instance, both the GOE and GUE satisfy the conditions above.    Another example is a random symmetric matrix whose upper-diagonal entries $\{\zeta_{ij} : 1 \leq i \leq j \leq n\}$ are independent Bernoulli random variables which take values $\pm 1$ with probability $1/2$.  

Let $W_n$ be given by \eqref{eq:Wn}, where $M_n$ is a $n \times n$ random Hermitian matrix that satisfies condition {\bf C0} and $\{D_n\}_{n \geq 1}$ is a sequence of deterministic diagonal matrices that satisfy the following condition.  

\begin{definition}[Condition {\bf C1}] \label{def:C1}
For each $n \geq 1$, let $D_n$ be a $n \times n$ deterministic real diagonal matrix.  The sequence of matrices $\{D_n\}_{n \geq 1}$ is said to obey condition {\bf C1} with probability measure $\mu_D$ if there exist constants $C''>0, l \geq 1, 0 < p_1, p_2, \ldots, p_l \leq 1$, and distinct $a_1, a_2, \ldots, a_l \in \R$ such that the following conditions hold.  
\begin{itemize}
\item The probability measure $\mu_D$ takes the form:
$$ \mu_D = \sum_{i=1}^l p_i \delta_{a_i}. $$
\item For each $n \geq 1$, there exist $0 \leq p_1^{(n)}, p_2^{(n)}, \ldots, p_l^{(n)} \leq 1$ such that
$$ \mu_{D_n} = \sum_{i=1}^l p_{i}^{(n)} \delta_{a_i}, $$
where 
$$ \sup_{1 \leq i \leq l} \left| p_{i}^{(n)} - p_{i} \right| \leq \frac{C''}{n} $$
for all $n \geq 1$.  
\end{itemize}
\end{definition}

In particular, it follows that if $\{D_n\}_{n \geq 1}$ satisfies condition {\bf C1} with measure $\mu_D$, then $\|D_n\|$ is uniformly bounded in $n$, and $\mu_{D_n} \longrightarrow \mu_{D}$ as $n \rightarrow \infty$.  In this case, the limiting ESD of $W_n$ has continuous density $\rho$, \cite{B}.  We assume the following conditions on $\rho$.  

\begin{definition}[Condition {\bf (A)}] \label{def:A}
We say that $\rho$ satisfies condition {\bf (A)} if $\rho$ is a continuous probability density supported on $q \geq 1$ disjoint intervals $[\alpha_j, \beta_j], j=1,\ldots,q$, where 
\begin{equation} \label{eq:condAnon0}
	\rho(x) > 0 \quad \text{for all} \quad x \in \bigcup_{j=1}^q (\alpha_j, \beta_j). 
\end{equation}
\end{definition}

For a density $\rho$ that satisfies condition {\bf (A)}, we define the quantiles
\begin{equation} \label{eq:def:svalues}
	s_j(\rho) := \int_{-\infty}^{\beta_j} \rho(x) dx, \quad j=1,\ldots, q, \quad s_0(\rho) := 0. 
\end{equation}
Given a small positive parameter $\eps$, we say that the index $i = i(n)$ is in the $(\eps,n)$-bulk of $\rho$ if 
$$ (s_{j-1}(\rho) +  \eps)n \leq i \leq (s_{j}(\rho) -\eps) n $$
for some $j=1,\ldots,q$.  We will also need the following definition.   

\begin{definition}[Moment matching]
We say that two complex random variables $\zeta$ and $\zeta'$ \textit{match to order $k$} if
$$ \E \Re(\zeta)^m\Im(\zeta)^l = \E \Re(\zeta')^m \Im(\zeta')^l $$
for all $m,l \geq 0$ such that $m+l \leq k$.  
\end{definition}

We will prove the following Four Moment theorem.

\begin{theorem}[Four Moment theorem] \label{thm:4moment}
There is a small positive constant $c_0$ such that for every $\eps > 0$ and $k \geq 1$ the following holds.  Let $M_n = (\zeta_{ij})_{1 \leq i,j \leq n}$ and $M_n' = (\zeta_{ij}')_{1 \leq i,j \leq n}$ be two random matrices satisfying condition {\bf C0}.  Assume that for any $1 \leq i < j \leq n$, $\zeta_{ij}$ and $\zeta_{ij}'$ match to order $4$ and for any $1 \leq i \leq n$, $\zeta_{ii}$ and $\zeta_{ii}'$ match to order $2$.  Let $\{D_n\}_{n \geq 1}$ satisfy condition {\bf C1} with measure $\mu_D$.  Assume the limiting density $\rho$ (defined by \eqref{eq:def:rho}) satisfies condition {\bf (A)}.  Set $A_n = \sqrt{n}M_n + n D_n$ and $A_n' = \sqrt{n} M_n' + n D_n$, and let $G:\R^k \rightarrow \R$ be a smooth function obeying the derivative bound
\begin{equation} \label{eq:der}
	|\nabla^j G(x)| \leq n^{c_0} 
\end{equation}
for all $0 \leq j \leq 5$ and $x \in \R^k$.  Then for any $ i_1 < i_2 < \cdots < i_k$ in the $(\eps,n)$-bulk of $\rho$, and for $n$ sufficiently large depending on $\eps,k,\mu_D$ (and the constants $C,C',C''$ in Definitions \ref{def:C0} and \ref{def:C1}) we have
\begin{equation} \label{eq:egcon}
|\E[G(\lambda_{i_1}(A_n), \ldots, \lambda_{i_k}(A_n))] - \E[G(\lambda_{i_1}(A_n'), \ldots, \lambda_{i_k}(A_n'))]| \leq n^{-c_0}. 
\end{equation}
If $\zeta_{ij}$ and $\zeta_{ij}'$ only match to order $3$ rather than $4$, then there is a positive constant $C$ independent of $c_0$ such that the conclusion \eqref{eq:egcon} still holds provided that one strengthens \eqref{eq:der} to
$$ |\nabla^j G(x)| \leq n^{-Cjc_0} $$
for all $0 \leq j \leq 5$ and $x \in \R^k$.  
\end{theorem}

As a consequence, the correlation functions are universal (under the four moments matching assumption).

\begin{theorem} [Universality of the correlation funcitons] \label{thm:corr} 
Let $M_n = (\zeta_{ij})_{1 \leq i,j \leq n}$ and $M_n' = (\zeta_{ij}')_{1 \leq i,j \leq n}$ be two random matrices satisfying condition {\bf C0}.  Assume that for any $1 \leq i < j \leq n$, $\zeta_{ij}$ and $\zeta_{ij}'$ match to order $4$ and for any $1 \leq i \leq n$, $\zeta_{ii}$ and $\zeta_{ii}'$ match to order $2$.  Let $\{D_n\}_{n \geq 1}$ satisfy condition {\bf C1} with measure $\mu_D$.  Assume the limiting density $\rho$ (defined by \eqref{eq:def:rho}) satisfies condition {\bf (A)}.  Let $f:\mathbb{R}^k \to \mathbb{C}$ be a continuous and compactly supported function.  Fix $k \geq 1$; let $\rho_n^{(k)}$ be the $k$-point correlation function of $W_n = \frac{1}{\sqrt{n}} M_n + D_n$, and let ${\rho'}_n^{(k)}$ be the $k$-point correlation function of $W_n' = \frac{1}{\sqrt{n}} M_n' + D_n$.  Fix $x_o \in \bigcup_{j=1}^q (\alpha_j, \beta_j)$, where $[\alpha_j, \beta_j], j=1,\ldots,q$ are the intervals from Definition \ref{def:A}.  Then 
\begin{align*}
	\int_{\mathbb{R}^k} f(u_1, \ldots, u_k) \left( \frac{1}{n \rho(x_0) } \right)^k \left(\rho_n^{(k)} - {\rho'}_n^{(k)} \right) \left(x_0 + \frac{u_1}{n \rho(x_0)}, \ldots, x_0 + \frac{u_k}{n \rho(x_0)} \right) du_1 \ldots du_k
\end{align*}
converges to zero as $n \rightarrow \infty$, where
$$ \left(\rho_n^{(k)} - {\rho'}_n^{(k)} \right) \left(v_1, \ldots, v_k \right) := \rho_n^{(k)}(v_1, \ldots, v_k) - {\rho'}_n^{(k)}(v_1, \ldots, v_k) $$
for any $v_1, \ldots, v_k \in \mathbb{R}$.  
\end{theorem} 

As an immediate consequence of Theorem \ref{thm:BKgauss} and Theorem \ref{thm:corr} (see also Remark \ref{rem:uniform}), we show that the sine kernel is universal.  

\begin{corollary} [Universality of the sine kernel] \label{cor:sine} 
Let $M_n = (\zeta_{ij})_{1 \leq i,j \leq n}$ be a random matrix that satisfies condition {\bf C0} and suppose $M_n' = (\zeta_{ij}')_{1 \leq i,j \leq n}$ is drawn from the GUE.  Assume that for any $1 \leq i < j \leq n$, $\zeta_{ij}$ and $\zeta_{ij}'$ match to order $4$ and for any $1 \leq i \leq n$, $\zeta_{ii}$ and $\zeta_{ii}'$ match to order $2$.  Let $D_n$ be a diagonal matrix with two eigenvalues $\pm a$ each with equal multiplicity and assume $a > 1$.  Let $\rho$ be the limiting density from Theorem \ref{thm:BKgauss}.  Let $f:\mathbb{R}^k \to \mathbb{C}$ be a continuous and compactly supported function.  Fix $k \geq 1$; let $\rho_n^{(k)}$ be the $k$-point correlation function of $W_n = \frac{1}{\sqrt{n}} M_n + D_n$.  Fix $x_o \in (-\alpha, -\beta) \cup (\beta, \alpha)$, where $\alpha, \beta$ are defined in Theorem \ref{thm:BKgauss}.  Then
\begin{align*}
	\int_{\mathbb{R}^k} & f(u_1, \ldots, u_k) \left( \frac{1}{n \rho(x_0) } \right)^k \rho_n^{(k)} \left(x_0 + \frac{u_1}{n \rho(x_0)}, \ldots, x_0 + \frac{u_k}{n \rho(x_0)} \right) du_1 \ldots du_k \\
	& \longrightarrow \int_{\mathbb{R}^k}  f(u_1, \ldots, u_k) \det \left( K(u_i,u_j) \right)_{i,j=1}^k du_1 \ldots du_k
\end{align*}
as $n \rightarrow \infty$, where $K(x,y) = \frac{ \sin \pi(x-y)}{\pi(x-y)}$.  
\end{corollary}

\begin{remark}
Corollary \ref{cor:sine} also holds when $0 < a < 1$.  This follows from Theorem \ref{thm:corr} and \cite[Theorem 1.2]{BKext2}.  In this case, the limiting density $\rho$ is supported on a single interval.  The value $a=1$ corresponds to a critical case because the limiting density is supported on two intervals for $a > 1$ and on a single interval for $0 < a < 1$.  In the case that $a=1$, the limiting density does not satisfy the assumptions of condition {\bf (A)}; see \cite{BKext3} for further details.  
\end{remark}

One way to improve upon the four moment assumption in Corollary \ref{cor:sine} (at least for matrices with complex entries) is to extend Theorem \ref{thm:BKgauss} to gauss divisible 
matrices, using some other methods (such as direct computation or local  relaxation flow). Let $\tilde M_n := (1- c_n)^{1/2} M_n +  c_n^{1/2} G_n$, where $G_n$ is drawn from the GUE and $0 < c_n <1$ is a parameter that may depend on $n$. If one can 
extend Theorem \ref{thm:BKgauss} to the model $\tilde M_n$ for constant $c_n$ (an analogue of Johansson's result \cite{joh}  in the Wigner case), then we can reduce the four moment matching assumption to
three matching moments (see \cite{EYTV,TVuniv}). If we can push $c_n$ smaller (say $n^{-\kappa}$ for a constant $\kappa>0$) as done by Erdos et. al. (see \cite{EY} for a survey), then using the arguments from \cite{TVmeh},  we can reduce the three matching moments to two matching moments, which is optimal. (This step would also  require a localization result in the spirit of \cite{EYY}, but this is probably not a big obstacle.) 

We also conjecture that Theorem \ref{thm:4moment} and Theorem \ref{thm:corr} should hold for more general diagonal matrices $D_n$.  We require condition {\bf C1} in order to establish a version of Pastur's relation \eqref{eq:prst} for the matrix $W_n$ when $\Im(z)$ is allowed to approach zero as $n$ tends to infinity; see Lemma \ref{lemma:stbound} and Remark \ref{rem:C1} for further details.

\subsection{Gaps between consecutive eigenvalues}

Theorem \ref{thm:4moment} is the analog of \cite[Theorem 15]{TVuniv} for Wigner random matrices with external source.  As in the proof of \cite[Theorem 15]{TVuniv}, we will need to verify that the eigenvectors of $W_n = \frac{1}{\sqrt{n}} M_n + D_n$ are completely delocalized.  That is, with high probability, the unit eigenvectors $u_1(W_n), \ldots, u_n(W_n)$ have all coordinates of magnitude $O(n^{-1/2})$, modulo logarithmic corrections.  For technical reasons we delay presenting this result until Section \ref{sec:eigenvectors}.  We will also need the following eigenvalue gap bound.  

\begin{theorem}[Eigenvalue gap bound] \label{thm:gap}
Let $M_n$ be a $n \times n$ matrix that satisfies condition {\bf C0}.  Let $\{D_n\}_{n \geq 1}$ satisfy condition {\bf C1} with measure $\mu_D$.  Assume the limiting density $\rho$ (defined by \eqref{eq:def:rho}) satisfies condition {\bf (A)}.  Set $A_n := \sqrt{n} M_n + n D_n$.  Let $\eps > 0$.  Then for every $c_0 > 0$ there exist $c_1>0$ and $n_0>1$ depending on $\eps,c_0, \mu_D$ (and the constants $C,C',C''$ in Definitions \ref{def:C0} and \ref{def:C1}) such that if $i$ is in the $(\eps,n)$-bulk of $\rho$, then
$$ \P ( \lambda_{i+1}(A_n) - \lambda_i(A_n) \leq n^{-c_0}) \leq n^{-c_1} $$
for all $n > n_0$.  
\end{theorem}

\subsection{Overview}

The paper is organized as follows.  In Section \ref{sec:tools}, we described our notation and introduce some tools we will need to prove our main results.  Section  \ref{sec:stability} contains our results concerning the stability of equation \eqref{eq:prst}. This stability 
is critical in all arguments using the Stieltjes transform to study the eigenvalues at a local scale. Some material in this section 
has appeared in earlier publications \cite{AGZ,B,K,V}.   
 Section \ref{sec:esd} is devoted to proving a local limit law for the ESD of $W_n$.  We also prove Theorem \ref{thm:corr} in Section \ref{sec:esd}.  Next, we show the eigenvectors of Wigner random matrices with external source are delocalized in Section \ref{sec:eigenvectors}. 
 Section \ref{sec:gap} contains the proof of Theorem \ref{thm:gap}.  Finally in Section \ref{sec:4moment}, we prove Theorem \ref{thm:4moment}.  
The approach here is more or less the same as in \cite{TVuniv}. However, due to the appearance of the external source,
we need to make  a number of delicate modification  and also introduce several new arguments.

\subsection*{Acknowledgements}
We are grateful to T. Tao for  assistance concerning Section \ref{sec:stability}.  We thank A. Kuijlaars for a useful e-mail exchange.  We also thank the anonymous referees for corrections.

\section{General tools and notation} \label{sec:tools}

In this section, we describe our notation and collect some tools we will need to prove Theorem \ref{thm:4moment}, Theorem \ref{thm:corr}, and Theorem \ref{thm:gap}.  

\subsection{Notation} \label{sec:notation}


We consider $n$ as an asymptotic parameter tending to infinity.  We use $X \ll Y$, $Y \gg X$, $Y=\Omega(X)$, or $X=O(Y)$ to denote the bound $|X| \leq CY$ for all sufficiently large $n$ for some constant $C$.  Notations such as $X \ll_k Y$, $X = O_k(Y)$ mean that the hidden constant $C$ depends on another constant $k$.  $X=o(Y)$ or $Y=\omega(X)$ means that $X/Y \rightarrow 0$ as $n \rightarrow \infty$.  The rate of decay here will be allowed to depend on other parameters.  We write $X=\Theta(Y)$ for $Y \ll X \ll Y$.  

It will be convenient to introduce the following notation for frequent events depending on $n$, in increasing order of likelihood.  

\begin{definition}[Frequent events]\label{freq-def}  Let $E$ be an event depending on $n$.
\begin{itemize}
\item $E$ holds \emph{asymptotically almost surely} if $\P(E) = 1-o(1)$.
\item $E$ holds \emph{with high probability} if $\P(E) \geq 1-O(n^{-c})$ for some constant $c>0$.
\item $E$ holds \emph{with overwhelming probability} if $\P(E) \geq 1-O_C(n^{-C})$ for \emph{every} constant $C>0$ (or equivalently, that $\P(E) \geq 1 - \exp(-\omega(\log n))$).
\item $E$ holds \emph{almost surely} if $\P(E)=1$.  
\end{itemize}
\end{definition}

We remind the reader that $W_n$ is the $n \times n$ Hermitian matrix defined by \eqref{eq:Wn}, where $M_n$ satisfies condition {\bf C0} and $\{D_n\}_{n \geq 1}$ satisfies condition {\bf C1} with measure $\mu_D$.  As discussed above the limiting ESD of $W_n$ has density, which we denote by $\rho$.  

For any interval $I \subset \mathbb{R}$, we define
\begin{equation} \label{eq:def:NI}
	N_I := n \mu_{W_n}(I) = \# \{ 1 \leq i \leq n : \lambda_i(W_n) \in I \}, 
\end{equation}
and let $|I|$ denote the length of $I$.  

We let $\sqrt{-1}$ denote the imaginary unit and reserve $i$ as an index.  We let $\#E$ and $|E|$ denote the cardinality of the set $E$.  For a matrix $A$, let $\|A\|$ denote the spectral norm of $A$; for a vector $v$, we let $\|v\|$ denote the Euclidian norm of $v$.  

\subsection{Tools}

We introduce the following lemmas from \cite{TVuniv, TVedge}. 

\begin{lemma}[Projection Lemma; \cite{TVedge}] \label{lemma:proj}
Let $X=(\xi_1, \ldots \xi_n) \in \C^n$ be a random vector whose entries are independent with mean zero, unit variance, and bounded in magnitude by $K$ almost surely for some $K$, where $K \geq 10(\E|\xi|^4+1)$.  Let $H$ be a subspace of dimension $d$ and $\pi_H$ the orthogonal projection onto $H$.  Then
$$ \P(|\|\pi_H(X)\| - \sqrt{d}| \geq t) \leq 10 \exp \left( - \frac{t^2}{10K^2} \right). $$
In particular, one has
$$ \|\pi_H(X)\| = \sqrt{d} + O(K \log n) $$
with overwhelming probability.  

The same conclusion holds (with $10$ replaced by another explicit constant) if one of the entries $\xi_j$ of $X$ is assumed to have variance $c$ instead of $1$, for some absolute constant $c>0$.  
\end{lemma}

\begin{lemma}[Tail bounds for complex random walks; \cite{TVuniv}] \label{lemma:tail}
Let $1 \leq N \leq n$ be integers, and let $A = (a_{ij})_{1 \leq i \leq N; 1 \leq j \leq n}$ be an $N \times n$ complex matrix whose $N$ rows are orthonormal in $\C^n$, and obeying the incompressibility condition
$$ \sup_{1 \leq i \leq N; 1 \leq j \leq n} |a_{ij}| \leq \sigma $$
for some $\sigma>0$.  Let $\zeta_1, \ldots, \zeta_n$ be independent complex random variables with mean zero, variance $\E|\zeta_j|^2$ equal to $1$, and obeying $\E|\zeta_i|^3 \leq C$ for some $C\geq 1$.  For each $1 \leq i \leq N$, let $S_i$ be the complex random variable
$$ S_i := \sum_{j=1}^n a_{ij} \zeta_j $$
and let $\vec{S}$ be the $\C^N$-valued random variable with coefficients $S_1, \ldots, S_N$.  
\begin{itemize}
\item (Upper tail bounds on $S_i$) For $t \geq 1$, we have 
$$ \P(|S_i| \geq t) \ll \exp(-ct^2) + C \sigma $$
for some absolute constant $c>0$.  
\item (Lower tail bounds on $\vec{S}$) For any $t \leq \sqrt{N}$, one has
$$ \P(\|\vec{S}\| \leq t) \ll O(t/\sqrt{N})^{\lfloor N/4 \rfloor} + CN^4 t^{-3} \sigma. $$
\end{itemize}
\end{lemma}

\subsection{$\eps$-nets}

We introduce $\eps$-nets as a convenient way to discretize a compact set.  Let $\eps > 0$.  A set $X$ is an $\eps$-net of a set $Y$ if for any $y \in Y$, there exists $x \in X$ such that $\|x-y\| \leq \eps$.  We will need the following estimate for the maximum size of an $\eps$-net.  

\begin{lemma} \label{lemma:epsnet}
Let $Q$ be a compact subset of $\{z \in \mathbb{C} : |z| \leq M \}$.  Then $Q$ admits an $\eps$-net of size at most
$$ \left( 1 + \frac{2M}{\eps} \right)^2. $$
\end{lemma}
\begin{proof}
Let $\mathcal{N}$ be maximal $\eps$ separated subset of $Q$.  That is, $|z-w| \geq \eps$ for all distinct $z,w \in \mathcal{N}$ and no subset of $Q$ containing $\mathcal{N}$ has this property.  Such a set can always be constructed by starting with an arbitrary point in $Q$ and at each step selecting a point that is at least $\eps$ distance away from those already selected.  Since $Q$ is compact, this procedure will terminate after a finite number of steps.  

We now claim that $\mathcal{N}$ is an $\eps$-net of $Q$.  Suppose to the contrary.  Then there would exist $z \in D$ that is at least $\eps$ from all points in $\mathcal{N}$.  In other words, $\mathcal{N} \cup \{ z \}$ would still be an $\eps$-separated subset of $Q$.  This contradicts the maximal assumption above.  

We now proceed by a volume argument.  At each point of $\mathcal{N}$ we place a ball of radius $\eps/2$.  By the triangle inequality, it is easy to verify that all such balls are disjoint and lie in the ball of radius $M + \eps/2$ centered at the origin.  Comparing the volumes give
$$ |\mathcal{N}| \leq \frac{ \left( M + \eps/2 \right)^2}{ (\eps/2)^2 } = \left( 1 + \frac{2M}{\eps} \right)^2. $$
\end{proof}

\subsection{Truncation}

Let $M_n = (\zeta_{ij})_{i,j=1}^n$ satisfy condition {\bf C0}.  From the union bound, we obtain
\begin{equation} \label{eq:truncation}
	\sup_{1 \leq i \leq j \leq n} |\zeta_{ij}| \leq \log^{C+1} n
\end{equation}
with overwhelming probability.  Following the truncation arguments in \cite{BSbook,TVuniv}, we may redefine the random variables $\zeta_{ij}$ on the events where their magnitude exceeds $\log^{C+1} n$.  It follows that it suffices to prove our main results under the assumption that \eqref{eq:truncation} holds almost surely.  As such, we allow the hidden constant in our asymptotic notation (such as $O,o,\ll$) to depend on the positive constants $C,C'$ from condition {\bf C0}.  

Similarly, following \cite{TVuniv}, we assume the $\zeta_{ij}$ have absolutely continuous distribution in the complex plane.  This assumption ensures that certain events (such as eigenvalue collision) occur with probability zero and can hence be safely ignored.  As in \cite{TVuniv}, the discrete case can be obtained from the continuous case by a standard limiting argument as well as Weyl's perturbation bound (see for instance \cite{Bma}).  Thus, we will henceforth assume the $\zeta_{ij}$ have absolutely continuous distribution in the complex plane and satisfy \eqref{eq:truncation} almost surely in order to prove Theorem \ref{thm:4moment} and Theorem \ref{thm:gap}.  

\section{Stability} \label{sec:stability}

In this section, we study the stability of equation \eqref{eq:prst}. Let us recall that in the Wigner case, the stability of
\eqref{eq:scst} (concerning the semicircle distribution) is straight-forward. On the other hand, with the presence of the external source, the situation is far more complex and requires a detailed proof. 

 Throughout this section, we let $\mu$ be a probability measure on the real line supported on a compact interval $[a,b]$, and let $m_\mu$ be its Stieltjes transform (defined in \eqref{eq:def:stmu}).  

For the remainder of the section, we utilize the following notation.  Let $R > 0$.  We use $z,z'$ to denote elements in the upper half-plane with $|z|,|z'| \leq R$.  We use $X \ll Y$, $Y \gg X$, $Y=\Omega(X)$, or $X=O(Y)$ to denote the bound $|X| \leq CY$ for some constant $C>0$ which may depend on $R$ and $\mu$, but which does not depend on any other parameters.    

We begin by analyzing the pointwise stability of the equation
$$ s = m_{\mu}(z+s).$$

\begin{lemma}[Stability]\label{stab}  
Let $z, s, s'$ be elements of the upper half-plane such that
\begin{equation}\label{szs}
s = m_\mu(z+s) + O(\eps)
\end{equation}
and
\begin{equation}\label{szs-2}
 s' = m_{\mu}(z+s'),
\end{equation}
with $|z| \leq R$ and some small $\eps>0$.  Suppose also that
$$ \Im z + \Im(s) \geq \eta > 0.$$
Then $s,s' = O(1)$ and
$$ s' = s + O(\eps^{1/3}) + O( \eps^{1/2}/\eta^{1/2} ).$$

The same conclusion holds if \eqref{szs-2} is weakened to $s' = m_{\mu}(z+s')+O(\eps)$, so long as one has the additional hypothesis $\Im z + \Im(s') \geq \eta$.
\end{lemma}

\begin{proof}  
We first show that $s=O(1)$.  Suppose for contradiction that $|s| \geq C$ for some large absolute constant $C$.  Since $|z| \leq R$, we then have $|z+s| \geq C/2$ (if $C$ is large enough), and hence $m_{\mu}(z+s) = O(1/C)$ (here we use the assumption that $\mu$ is supported on a compact set).  But then the right-hand side of \eqref{szs} is $O(1/C) + O(\eps)$, which contradicts the hypothesis $|s| \geq C$ for $C$ large enough.  Thus, $s=O(1)$, and similarly $s'=O(1)$.

Write $w := z+s$ and $w' := z+s'$.  Then $w,w'$ are in the upper half-plane.  Our task is now to show that
$$ |w' - w| \ll \eps^{1/3} + \eps^{1/2} / \eta^{1/2}.$$
We may of course assume that $w' \neq w$, as the claim is trivial otherwise.
Subtracting \eqref{szs} from \eqref{szs-2}, we obtain
$$ m_{\mu}(w') - m_{\mu}(w) = w' - w + O(\eps).$$
But from \eqref{eq:def:stmu}, we have
$$
 m_{\mu}(w')-m_{\mu}(w) = (w'-w) \int_\R \frac{d\mu(x)}{(x-w)(x-w')}$$
 and hence
\begin{equation}\label{dmux}
\int_\R \frac{d\mu(x)}{(x-w)(x-w')} = 1 + O\left( \frac{\eps}{|w-w'|} \right).
 \end{equation}
On the other hand, comparing the imaginary parts for both sides of \eqref{szs}, we obtain
$$ \Im m_{\mu}(w) \leq \Im w + O(\eps).$$
From \eqref{eq:def:stmu}, we have
$$ \Im m_{\mu}(w) = (\Im w) \int_\R \frac{d\mu(x)}{|x-w|^2}$$
and hence
$$ \int_\R \frac{d\mu(x)}{|x-w|^2} \leq 1 + O\left( \frac{\eps}{\Im(w)} \right).$$
Since $w=z+s$, we have $\Im(w) \geq \eta$ and thus
$$ \int_\R \frac{d\mu(x)}{|x-w|^2} \leq 1 + O\left( \frac{\eps}{\eta} \right).$$
Similarly (noting that there is no error term in \eqref{szs-2})
$$ \int_\R \frac{d\mu(x)}{|x-w'|^2} \leq 1.$$
Now we return to \eqref{dmux}.  From the arithmetic mean-geometric mean inequality, we have
$$  \left| \frac{1}{(x-w)(x-w')} \right| \leq \frac{1}{2} \frac{1}{|x-w|^2} + \frac{1}{2} \frac{1}{|x-w'|^2}.$$
Since $w \neq w'$, it follows that
$$ \left|\Re \left[ \frac{1}{(x-w)(x-w')} \right] \right| = (1-\delta)\left( \frac{1}{2} \frac{1}{|x-w|^2} + \frac{1}{2} \frac{1}{|x-w'|^2} \right)$$
for some $\delta > 0$.  Then we have
$$ |x-w| = (1+O(\delta)) |x-w'|. $$
and
$$ \angle(x-w, x-w') = O( \delta^{1/2} ).$$
Since $x,w,w' = O(1)$, we obtain $w-w' = O(\delta^{1/2})$.  We conclude that
\begin{equation}\label{rew}
 \Re \left[\frac{1}{(x-w)(x-w')} \right] \leq (1-c|w-w'|^{2}) \left(\frac{1}{2} \frac{1}{|x-w|^2} + \frac{1}{2} \frac{1}{|x-w'|^2} \right)
\end{equation}
for some $c>0$, and hence
$$
\Re \int_\R \frac{d\mu(x)}{(x-w)(x-w')} \leq (1-c|w-w'|^{2})  \left(1 + O\left( \frac{\eps}{\eta} \right) \right).$$
Comparing this against \eqref{dmux} yields
$$ |w-w'|^{2} \ll \frac{\eps}{\eta} + \frac{\eps}{|w-w'|}, $$
and the claim follows.

A similar argument works when there is an error term in \eqref{szs-2}.
\end{proof}

Unfortunately, pointwise stability breaks down when $z$ approaches the real axis.  For instance, let $\mu = \frac{1}{2} \delta_{\sqrt{2}} + \frac{1}{2} \delta_{-\sqrt{2}}$, so that
$$ m_{\mu}(z) = \frac{1}{2} \frac{1}{\sqrt{2}-z} + \frac{1}{2} \frac{1}{-\sqrt{2}-z} = \frac{z}{2-z^2}.$$
Then the equation $s=m_{\mu}(z+s)$ has a triple root $s=\pm 1,0$ when $z=0$, and so no naive analogue of the above lemma works in the limit when $z \to 0$.  In Theorem \ref{thm:stab} below, we use a continuity argument (rather than a pointwise argument) to obtain the desired stability.

\begin{corollary}  \label{unique} 
There exists a unique holomorphic function $z \mapsto s_0(z)$ from the upper half-plane to the upper half-plane such that $s_0(z) = m_{\mu}(z+s_0(z))$ for all $z$ in the upper half-plane.
\end{corollary}

\begin{proof} 
Uniqueness follows by sending $\eps \to 0$ in Lemma \ref{stab}.  For existence, observe that the function $w \mapsto w-m_{\mu}(w)$ is asymptotic to $w+O(1)$ for large $w$, and takes values in the lower half-plane when $w$ is real.  A winding number argument then shows that for every $z$ in the upper half-plane, there is a $w$ in the upper half-plane with $w-m_{\mu}(w)=z$; setting $s_0(z) := m_{\mu}(w)$, we conclude that $s_0(z)=m_{\mu}(z+s_0(z))$.  From Lemma \ref{stab} it follows that $w$ and $s_0$ depend continuously on $z$.  As $w \mapsto w-m_{\mu}(w)$ is holomorphic and invertible on the range of $w$ (which is open, by the invariance of domain theorem), we conclude that $w-m_{\mu}(w)$ has no stationary points (this can also be obtained from a variant of Lemma \ref{stab}) and so the inverse map $z \mapsto w$ is holomorphic, and the claim follows.\footnote{One can also prove existence by using the free convolution of $\mu$ with the semicircle law.  See, for instance, \cite{AGZ,B,K,V} and references therein.} 
\end{proof}

It is easy to see that this function $z \mapsto s_0(z)$ is asymptotic to $-\frac{1}{z}$ as $|z| \to\infty$, and so by the Herglotz representation theorem (see for instance \cite{AK}), it is the Stieltjes transform of some probability measure $\nu$.

We now verify that the function $z \mapsto s_o(z)$ is H\"older continuous.  

\begin{lemma}[H\"older continuity]\label{hold}  
For $z,z'$ in the upper half-plane with $|z|, |z'| \leq R$, we have
$$ |s_0(z)-s_0(z')| \ll |z-z'|^{1/3}.$$
\end{lemma}

\begin{proof}  
Write $w = z + s_0(z)$ and $w' = z' + s_0(z')$.  Then
$$ w - m_{\mu}(w) = z$$
and
$$ w' - m_{\mu}(w') = z'.$$
Subtracting as in the proof of Lemma \ref{stab}, we conclude that
$$ \int_\R \frac{d\mu(x)}{(x-w)(x-w')} = 1 + O \left( \frac{|z-z'|}{|w-w'|} \right).$$
Also, taking imaginary parts as in the proof of Lemma \ref{stab}, we have
$$ \int_\R \frac{d\mu(x)}{|x-w|^2} \leq 1$$
and
$$ \int_\R \frac{d\mu(x)}{|x-w'|^2} \leq 1.$$
Using \eqref{rew}, we obtain the claim.
\end{proof}

\begin{remark}
Note that the Stieltjes transform of the semicircle law is in fact H\"older continuous of order $1/2$ rather than $1/3$, so some improvement to Lemma \ref{hold} is presumably possible, but we will not pursue this.
\end{remark}

Now we investigate the stability of $s=m_{\mu}(z+s)$ near the real line.  

\begin{theorem}  \label{thm:stab} 
Let $R>0$ and $\eps > 0$.  Then there exists $\delta > 0$ such that if $N>1$ and $s$ is the Stieltjes transform of a probability measure which satisfies
$$ |s(z)-m_{\mu}(z+s(z))| \leq \delta$$
for all $z$ with $\Im(z)>1/N$ and $|z| \leq R$, then
$$ |s(z)-s_0(z)| \leq \eps$$
for all $z$ with $\Im(z)>1/N$ and $|z| \leq R$.
\end{theorem}

\begin{proof}  
Suppose this is not the case.  Then there exists a sequence $\delta_n>0$ tending to zero, a sequence $s_n$ of Stieltjes transforms of probability measures, a sequence $N_n \geq 1$, and a sequence $z_n$ of complex numbers with $\Im(z_n)>1/N_n$ and $|z| \leq R$, such that
$$ |s_n(z)-m_{\mu}(z+s_n(z))| \leq \delta_n$$
for all $z$ with $\Im(z)>1/N_n$ and $|z| \leq R$, but
\begin{equation}\label{soo}
 |s_n(z_n)-s_0(z_n)| > \eps.
\end{equation}
By passing to a subsequence, we may assume that $z_n$ converges to a limit $z_\infty$.  From Lemma \ref{stab}, both $s_n(z_n)$ and $s_0(z_n)$ are uniformly bounded, so we may assume that they converge to some limits $s_\infty$ and $s^0_\infty$ respectively, with $|s_\infty-s^0_\infty| \geq \eps$.

For the remainder of the proof, we view $n$ as an asymptotic parameter tending to infinity.  We use $X = o(1)$ to mean $X \rightarrow 0$ as $n \rightarrow 0$.  By construction, we have
$$ s_n(z) = m_{\mu}(z+s_n(z)) + o(1)$$
and
$$ s_0(z) = m_{\mu}(z+s_0(z))$$
for all $z$ with $\Im(z)>1/N_n$ and $|z| \leq R$.  Applying Lemma \ref{stab}, we conclude the existence of some $\eta_n > 0$ with $\eta_n \to 0$ such that 
$$ s_n(z) = s_0(z) + o(1)$$
whenever $\Im(z) + \Im(s_n(z)) \geq \eta_n$ and $|z| \leq R$.  To put it another way, we have the following dichotomy: whenever $\Im(z)>1/N_n$ and $|z| \leq R$, then at least one of
\begin{equation}\label{snop}
 s_n(z) = s_0(z) + o(1)
\end{equation}
and
\begin{equation}\label{snop-2}
 \Im(z), \Im(s_n(z)) \leq \eta_n = o(1)
 \end{equation}
are true.

This implies that
$$\Im(z_n) + \Im(s_n(z_n)) = o(1);$$
in particular, $z_\infty$ and $s_\infty$ are real, and $N_n$ must go to infinity as $n \to \infty$.

Seeking a contradiction, suppose that $s_\infty^0$ is not real.  Then for $n$ sufficiently large, $\Im (s_0(z_n))$ is bounded away from zero, and thus by H\"older continuity $\Im(s_0(z_n + \sqrt{-1} t \eta_n))$ is also bounded away from zero for $0 \leq t \leq 1$.  But from the dichotomy, we have $s_n(z_n+\sqrt{-1}\eta_n) = s_0(z_n + \sqrt{-1} \eta_n) + o(1)$. Thus, by a continuity argument, we have $s_n(z_n+\sqrt{-1}t\eta_n) = s_0(z_n + \sqrt{-1} \eta_n) + o(1)$ for all $0 \leq t \leq 1$ (since there is no overlap between the two choices \eqref{snop}, \eqref{snop-2} of the dichotomy), and so $s_n(z_n) = s_0(z_n)+o(1)$, contradicting \eqref{soo}.  We conclude that $s_\infty^0$ is real.

Let $I$ be the interval in $\R$ between $w_\infty^0 := s_\infty^0 + z_\infty$ and $w_\infty := s_\infty + z_\infty$.  By \eqref{soo}, we have $|s_\infty-s_\infty^0| \geq \eps$ and hence $I$ has positive length.

Suppose first that $\mu$ is not entirely vanishing on the interior of $I$.  Then by compactness, there must exist a point $x$ in the interior of $I$ which is a point of density for $\mu$ in the sense that $\liminf_{r \to 0} \frac{1}{2r} \mu([x-r,x+r]) > 0$.  This implies that $m_{\mu}(x+\sqrt{-1}r)$ is bounded away from zero for sufficiently small $r$.

On the other hand, we have $s_n(z_n) = s_\infty + o(1)$, and from the dichotomy (and H\"older continuity) we also have $s_n(z_n+\sqrt{-1} \eta_n) = s_\infty^0 + o(1)$.  The dichotomy (and H\"older continuity) also shows that for all $0 \leq t \leq 1$ we have $\Im s_n(z_n+\sqrt{-1}t\eta_n) = o(1)$.  Thus, by the intermediate value theorem, there must exist some $0 \leq t_n \leq 1$ such that $\Re (z_n + s_n(z_n+\sqrt{-1}t_n \eta_n)) = x$, thus
$$z_n + s_n(z_n+\sqrt{-1}t_n \eta_n) = x + \sqrt{-1}r_n$$
for some $r_n=o(1)$.  But then on applying $m_{\mu}$, we see that $\Im s_n(z_n+\sqrt{-1}t\eta_n)$ is bounded away from zero, a contradiction.

Thus, we may assume that $\mu$ is vanishing on the interior of $I$.  This implies $m_{\mu}$ is complex analytic in a neighborhood of the interior.  As before, we have that for all $x \in I$, we can find $t_n$ such that
$$w_n := z_n + s_n(z_n+\sqrt{-1}t_n \eta_n) = x + o(1).$$
Since $w_n - m_{\mu}(w_n) = z_n +o(1) = z_\infty + o(1)$, we conclude that $x-m_{\mu}(x) = z_\infty+o(1)$, and thus on taking limits one has
$$ x-m_{\mu}(x) = z_\infty$$
for all $x$ in the interior of $I$.  But then by analytic continuation this forces $w-m_{\mu}(w) = z_\infty$ for all $w$ in the upper half-plane, which is absurd.
\end{proof}

\section{Asymptotics for the ESD and the proof of Theorem \ref{thm:corr}} \label{sec:esd}

This section is devoted to proving a local limit law for $\mu_{W_n}$.  We begin by introducing some new notation.  We let $m_n$ denote the Stieltjes transform of $W_n$.  That is, 
$$ m_n(z) := m_{W_n}(z) =  \int_{\mathbb{R}} \frac{ d \mu_{W_n}(x)}{x - z}, \quad z \in \mathbb{C}^+. $$
Under the assumption that $\{D_n\}_{n \geq 1}$ satisfies condition {\bf C1} with measure $\mu_D$, it follows from our discussion above that $m_n(z) \rightarrow m(z)$ almost surely as $n \rightarrow \infty$ for each $z \in \mathbb{C}^+$.  Here $m$ is the Stieltjes transform of a probability measure with continuous density $\rho$.  In particular, 
$$ m(z) := \int_{\mathbb{R}} \frac{\rho(x)}{x - z} dx, \quad z \in \mathbb{C}^+. $$
Furthermore $m$ is the unique Stieltjes transform that satisfies \eqref{eq:prst} for all $z$ in the upper half plane.  

We let $g_n$ denote the Stieltjes transform of $\mu_{D_n}$.  In other words,
$$ g_n(z) := \int_{\mathbb{R}} \frac{ d \mu_{D_n}(x) }{ x - z }, \quad z \in \mathbb{C}^+. $$
If $\{D_n\}_{n \geq 1}$ satisfies condition {\bf C1} with measure $\mu_D$, we let $g$ denote the Stieltjes transform of $\mu_D$.  In this case, it follows that $g_n(z) \rightarrow g(z)$ as $n \rightarrow \infty$ for each $z \in \mathbb{C}^+$.  

In general, the entries of $D_n$ and $M_n$ are allowed to depend on $n$.  We will typically deal with the case when $n$ is large and fixed.  In this case, we drop the dependence on $n$ from our notation; we write $M_n = \{\zeta_{ij}\}_{i,j=1}^n$, and let $d_{11}, d_{22}, \ldots, d_{nn}$ denote the diagonal entries $D_n$.  

We remind the reader that $N_I$ is defined by \eqref{eq:def:NI}.  The results in this section are similar to \cite[Theorem 2.18]{LS}.  We begin with the following crude bound. 

\begin{lemma}[Upper bound on the ESD] \label{lemma:crude}
Let $M_n$ be a random Hermitian matrix that satisfies condition {\bf C0} whose entries are bounded in magnitude by $K$ almost surely for some $K \geq 1$.  Let $\{D_n\}_{n \geq 1}$ satisfy condition {\bf C1} with measure $\mu_D$.  Set $W_n := \frac{1}{\sqrt{n}} M_n + D_n$.  Then for any interval $I \subset \R$ with $|I| \geq \frac{K^2 \log^2 n}{n}$,
$$ N_I \ll n|I| $$
with overwhelming probability.  
\end{lemma}

\begin{proof}
By applying the union bound, it is enough to consider the case when $|I| = \frac{K^2 \log^2 n}{n} = : \eta$.  Let $x$ be the center of $I$.  Set $z = x + \sqrt{-1} \eta$.  Define the event
$$ \Omega_n := \{ N_I \geq C n \eta \} \cap \{ \Im(m_n(x+\sqrt{-1}\eta)) \geq C \}. $$
Since $\{ \Im(m_n(x+\sqrt{-1}\eta)) \geq C \} \subset \{ N_I \geq C n \eta \}$, it suffices to show that the event $\Omega_n$ fails with overwhelming probability.  

On the event $\Omega_n$,
\begin{align*}
	C \leq \Im(m_n(x+\sqrt{-1}\eta)) &\leq \frac{1}{n} \sum_{k=1}^n \Im \left(\frac{1}{\frac{1}{\sqrt{n}}\zeta_{kk}+d_{kk} -z - Y_k} \right) \\
		& \leq \frac{1}{n} \sum_{k=1}^n \frac{1}{| \eta + \Im(Y_k)|}
\end{align*}
where $Y_k = a_k^\ast (W_{n,k} - zI)^{-1} a_k$, $W_{n,k}$ is the matrix $W_n$ with the $k$-th row and column removed, and $a_k$ is the $k$-th column of $W_n$ with the $k$-th entry removed.  We can then write
\begin{align*}
	\Im(Y_k) &= \Im \left( \sum_{j=1}^{n-1} \frac{| a_k \cdot u_j(W_{n,k})|^2}{\lambda_j(W_{n,k}) - z} \right) \\
		& = \eta \sum_{j=1}^{n-1} \frac{| a_k \cdot u_j(W_{n,k})|^2}{(\lambda_j(W_{n,k})-x)^2 + \eta^2 }
\end{align*}
where $u_j(W_{n,k})$ is the eigenvector of $W_{n,k}$ corresponding to eigenvalue $\lambda_j(W_{n,k})$.  

On $\Omega_n$, there exists $i_-$ and $i_+$ such that $i_+ - i_- \geq C n \eta$ and $\lambda_i(W_n) \in I$ for $i_- \leq i \leq i_+$.  By Cauchy's Interlacing Theorem, $\lambda_i(W_{n,k}) \in I$ for $i_- < i < i_+$.  So
\begin{align*}
	\Im(Y_k) \geq \eta \sum_{i_- < i < i_+} \frac{|a_k \cdot u_i(W_{n,k})|^2}{ \left( \frac{\eta}{2} \right)^2 + \eta^2 } \geq \frac{4}{5} \frac{1}{\eta} \sum_{i_- < i < i_+} |a_k \cdot u_i(W_{n,k})|^2 \geq \frac{4}{5} \frac{1}{\eta} \| P_{H_k}(a_k) \|^2
\end{align*}
where $P_{H_k}$ is the orthogonal projection on to the $(i_+-i_- -2)$-dimensional subspace $H_k$ spanned by $u_i(W_{n,k})$ for $i_- < i < i_+$.  

Combining the bounds above, we obtain
$$ C \leq \frac{1}{n} \sum_{k=1}^n \frac{1}{\eta + \frac{4}{5} \frac{1}{\eta} \| P_{H_k}(a_k) \|^2 } = \frac{1}{n} \sum_{k=1}^n \frac{\eta}{\eta^2 + \frac{4}{5} \|P_{H_k}(a_k) \|^2 }. $$
By Lemma \ref{lemma:proj} and the union bound over $O(n^2)$ choices of $i_+$ and $i_-$, we have that $\|P_{H_k}(a_k) \|^2 = \Omega(\eta)$ with overwhelming probability.  Therefore, by taking $C$ large enough, it follows that the event $\Omega_n$ fails with overwhelming probability.  
\end{proof}

\begin{lemma}[Control of Stieltjes transform implies control of ESD] \label{lemma:control}
Let $M_n$ satisfy condition {\bf C0}.  Let $\{D_n\}_{n \geq 1}$ satisfy condition {\bf C1} with measure $\mu_D$.  Assume that the limiting density $\rho$ is bounded.  Let $1/10 \geq \eta \geq 1/n$ and $L>1$.  Then for any $0 < \eps < 1/2$, there exists $\delta > 0$ such that the following holds.  If
\begin{equation} \label{eq:snzg}
	|m_n(z) - g(z+m_n(z))| \leq \delta 
\end{equation}
with (uniformly) overwhelming probability for all $z$ with $|z| \leq L$ and $\Im(z) \geq \eta$, then for any interval $I$ in $[-L+1/2,L-1/2]$ with $|I| \geq \max(2 \eta, \frac{\eta}{\eps} \log \frac{1}{\eps})$, one has
$$ \left| N_I - n \int_{I} \rho(x)dx \right| \ll \eps n |I| $$
with overwhelming probability.  
\end{lemma}

\begin{proof}
Let $\eps>0$.  By Theorem \ref{thm:stab} in Section \ref{sec:stability}, there exists $\delta > 0$ such that \eqref{eq:snzg} implies
\begin{equation} \label{eq:snszeps}
	|m_n(z) - m(z)| \leq \eps
\end{equation}
with (uniformly) overwhelming probability for all $z$ with $|z| \leq L$ and $\Im(z) \geq \eta$.  It follows from Lemma \ref{stab} in Section \ref{sec:stability} that $|m_n(x+\sqrt{-1}\eta)| \ll 1$ with overwhelming probability for $-L+1/2 \leq x \leq L-1/2$.  Thus, we have
$$ \Im(m_n(x + \sqrt{-1} \eta)) = \frac{1}{n} \sum_{i=1}^n \frac{\eta}{\eta^2 + (\lambda_i(W_n) -x)^2 } \ll 1 $$
with overwhelming probability for all $-L+1/2 \leq x \leq L-1/2$.  Thus, for $I \subset [-L+1/2,L-1/2]$ with $|I| = \eta$, we have that $N_I \ll n \eta$ with overwhelming probability.  By summing over $I$, we obtain
\begin{equation} \label{eq:nni}
	N_I \ll n |I|
\end{equation}
with overwhelming probability for $I \subset [-L+1/2,L-1/2]$ with length $|I| \geq \eta$.  

Let $I \subset [-L+1/2, L-1/2]$ with $|I| \geq \max(2 \eta, \frac{\eta}{\eps} \log \frac{1}{\eps})$.  Define the function
$$ F(y) := n^{-100} \sum_{x \in I; n^{-100}|x} \frac{\eta}{\pi(\eta^2 + |y-x|^2)} $$
where the sum is over all $x \in I$ that are multiples of $n^{-100}$.  We then have that
$$ \frac{1}{n} \sum_{j=1}^n F(\lambda_j(W_n)) = n^{-100}\frac{1}{\pi} \Im \sum_{x \in I; n^{-100}|x} m_n(x+\sqrt{-1}\eta) $$
and
$$ \int_\R F(y) \rho(y) dy = n^{-100}\frac{1}{\pi} \Im \sum_{x \in I; n^{-100}|x} m(x+\sqrt{-1}\eta). $$

By \eqref{eq:snszeps}, $m_n(x+\sqrt{-1}\eta) = m(x+\sqrt{-1}\eta) + O(\eps)$ with overwhelming probability for all $x \in I$ and hence
$$ \frac{1}{n} \sum_{j=1}^n F(\lambda_j(W_n)) = \int_\R F(y) \rho(y) dy + O(\eps|I|)  $$
with overwhelming probability.  

By Riemann integration,
\begin{equation} \label{eq:inteqF}
	F(y) = \int_I \frac{\eta}{\pi (\eta^2 + (y-x)^2)}dx + O(n^{-10}). 
\end{equation}
From \eqref{eq:inteqF}, we obtain the following pointwise bounds for $y \notin I$:
$$ F(y) \ll \frac{1}{1 + \frac{\dist(y,I)}{\eta}} + n^{-10} $$
when $\dist(y,I) \leq |I|$, and
$$ F(y)  \ll \frac{\eta |I|}{\dist(y,I)^2} + n^{-10} $$
when $\dist(y,I) > |I|$.  

Since
$$ \int_\R \frac{\eta}{\pi (\eta^2 + (y-x)^2)}dx = 1, $$
it follows that
$$ F(y) = 1 + O\left( \frac{1}{1 + \frac{\dist(y,I^c)}{\eta} } \right) + O(n^{-10}) $$
when $y \in I$.  

Thus, we obtain
$$ \int_\R F(y) \rho(y) dy = \int_I \rho(y) dy + O \left( \eta \log \left( \frac{|I|}{\eta} \right) \right).  $$
Similarly, using Riemann integration, \eqref{eq:nni}, and the trivial bound $N_J \leq n$ for $J$ outside $[-L+1/2,L-1/2]$, we obtain
$$ \frac{1}{n} \sum_{i=1}^n F(\lambda_i(W_n)) = \frac{1}{n} N_I + O \left( \eta \log \left( \frac{|I|}{\eta} \right) \right) $$
with overwhelming probability.  Therefore, we conclude that
$$ N_I = n \int_I \rho(y) dy + O (\eps n |I|) $$
since $|I| \geq \frac{\eta}{\eps} \log \frac{1}{\eps} $.  
\end{proof}

\begin{lemma}[Concentration of Stieltjes Transform up to the edge] \label{lemma:stbound}
Let $M_n$ satisfy condition {\bf C0}.  Suppose the entries of $M_n$ are bounded by $K$ for some $K \geq 1$ where $\frac{K}{\sqrt{n}} \ll \frac{1}{\log^3 n}$.  Let $\{D_n\}_{n \geq 1}$ satisfy condition {\bf C1} with measure $\mu_D$.  Then
$$ |m_n(z) - g(z+m_n(z))| \ll_L \frac{1}{\log n} $$
with (uniformly) overwhelming probability for all $z$ with $|z| \leq L$ and 
$$ \Im(z) \geq \frac{K^2 \log^{20} n}{n}. $$
\end{lemma}

\begin{proof}
By Schur's complement, for any $z \in \C$ with $\Im(z) > 0$,
\begin{equation} \label{eq:snzeta}
	m_n(z) = \frac{1}{n} \sum_{k=1}^n \frac{1}{\frac{1}{\sqrt{n}} \zeta_{kk} + d_{kk} - z - Y_k} 
\end{equation}
where 
$$ Y_k = a_k^\ast (W_{n,k} - zI)^{-1} a_k, $$
$W_{n,k}$ is the matrix $W_n$ with the $k$-th row and column removed, and $a_k$ is the $k$-th column of $W_{n}$ with the $k$-th entry removed.  Write $z = x + \sqrt{-1}\eta$ where
$$ \eta n \geq K^2 \log^{20} n. $$
Then
$$ \E[Y_k|W_{n,k}] = \frac{1}{n} \tr (W_{n,k} - zI)^{-1} = (1-\frac{1}{n}) m_{n,k}(z), $$
where 
$$ m_{n,k}(z) := \frac{1}{n-1} \sum_{i=1}^{n-1} \frac{1}{\lambda_i(W_{n,k}) - z}. $$
By Cauchy's interlacing law,
$$ m_n(z) - (1-\frac{1}{n})m_{n,k}(z) \ll \frac{1}{n \eta} \ll \frac{1}{\log^{20} n}. $$

We will now show that with overwhelming probability
$$ Y_k - E[Y_k|W_{n,k}] \ll \frac{1}{\log^3 n}. $$
In fact, by rewriting $Y_k$, it suffices to show that
\begin{equation} \label{eq:rjshow}
\sum_{j=1}^{n-1} \frac{R_j}{\lambda_j(W_{n,k}) - x - \sqrt{-1} \eta} \ll \frac{1}{\log^3 n}
\end{equation}
with overwhelming probability, where
$$ R_j := |u_j(W_{n,k})^\ast a_k|^2 - \frac{1}{n}. $$
Let $1 \leq i_- < i_+ \leq n$, then
$$ \sum_{i_- \leq j \leq i_+} R_j = \|P_H a_k \|^2 - \frac{\dim(H)}{n} $$
where $H$ is the space spanned by the $u_j(W_{n,k})$ for $i_- \leq j \leq i_+$.  From Lemma \ref{lemma:proj} and the union bound, we have that
\begin{equation} \label{eq:rj1}
	\left| \sum_{i_- \leq j \leq i_+} R_j \right| \ll \frac{ \sqrt{i_+ - i_i} K \log n + K^2 \log^2 n}{n} 
\end{equation}
with overwhelming probability.  It follows that
$$ \|P_H a_k \|^2 \ll \frac{i_+ - i_-}{n} + \frac{ \sqrt{i_+ - i_i} K \log n + K^2 \log^2 n}{n} $$
and hence by the triangle inequality
\begin{equation} \label{eq:rj2}
	\sum_{i_- \leq j \leq i_+} |R_j| \ll \frac{i_+ - i_-}{n} + \frac{K^2 \log^2 n}{n} 
\end{equation}
with overwhelming probability.  

We now study the sum on the left-hand side of \eqref{eq:rjshow}.  We classify the indices $j$ according to the distance $|\lambda_j(W_{n,k}) - x|$.  To begin, we define
\begin{align*}
	\delta' := \frac{1}{\log^{3.01} n}, \quad \alpha := \frac{1}{\log^{4.01} n}.
\end{align*}
Let
$$ B = \{ j : |\lambda_j(W_{n,k}) - x| \leq \delta' \eta \}. $$
By Lemma \ref{lemma:crude}, $|B| \ll \delta' \eta n$ with overwhelming probability.  For each $j \in B$, 
$$ \frac{1}{\lambda_j(W_{n,k}) - x - \sqrt{-1} \eta} = O \left( \frac{1}{\eta} \right). $$
Thus, by \eqref{eq:rj2}, we obtain 
$$ \sum_{j \in B} \frac{R_j}{\lambda_j(W_{n,k}) - x - \sqrt{-1} \eta} \ll \frac{1}{\eta} \frac{\delta' \eta n}{n} \ll \frac{1}{\log^3 n} $$
with overwhelming probability.  

Next we consider the following set of indices:
$$ A_l = \{ j : (1+\alpha)^l \delta ' \eta < | \lambda_j(W_{n,k}) - x| \leq (1+\alpha)^{l+1} \delta' \eta \} $$
for some integers $0 \leq l \ll \log n/\alpha$.  By Lemma \ref{lemma:crude}, $|A_l| \ll (1+\alpha)^l \alpha \delta' \eta n$ with overwhelming probability.  On $A_l$, the quantity
$$ \frac{1}{\lambda_j(W_{n,k}) - x - \sqrt{-1} \eta} $$
has magnitude $O \left( \frac{1}{(1+\alpha)^l \delta' \eta} \right)$ and fluctuates by $O \left( \frac{\alpha}{(1+\alpha)^l \delta' \eta} \right)$.  So by \eqref{eq:rj1} and \eqref{eq:rj2}, we obtain, with overwhelming probability,
$$ \sum_{j \in A_l} \frac{R_j}{\lambda_j(W_{n,k}) - x - \sqrt{-1} \eta} \ll \alpha(1+\alpha)^{-l/2} \frac{K \log n}{\sqrt{\alpha \delta' \eta n}} + \alpha^2. $$
Summing over $l$, we obtain
$$ \sum_{l} \sum_{j \in A_l} \frac{R_j}{\lambda_j(W_{n,k}) - x - \sqrt{-1} \eta} \ll \frac{K \log n}{\sqrt{\alpha \delta' \eta n}} + \alpha \log n $$
with overwhelming probability.  

Combining the bounds above, we obtain \eqref{eq:rjshow}.  By assumption, we have that
$$ \frac{|\zeta_{kk}|}{\sqrt{n}} \leq \frac{K}{\sqrt{n}} \ll \frac{1}{\log^3 n}. $$
Therefore, by the union bound and \eqref{eq:snzeta}, we obtain
\begin{equation} \label{eq:ndkksum}
	m_n(z) = \frac{1}{n} \sum_{k=1}^n \frac{1}{d_{kk} - z - m_n(z) + O\left( \frac{1}{\log^3 n} \right)}
\end{equation}
with overwhelming probability.  We remark that the implicit constant in the asymptotic notation of equation \eqref{eq:ndkksum} is independent of $k$.  

Let $Q$ be the region in the complex plane where
$$ |z| \leq L, \quad \Im(z) \geq \frac{K^2 \log^{20} n}{n}. $$
We apply an $\eps$-net argument to extend \eqref{eq:ndkksum} to the entire region $Q$.  Indeed, by the resolvent identity, for $z,z' \in Q$, we have
$$ |m_n(z) - m_n(z')| \leq \frac{|z-z'|}{\Im(z) \Im(z')} \leq n^{2} |z-z'|. $$
Thus, by a standard $\eps$-net argument with (say) $\eps:=n^{-100}$, Lemma \ref{lemma:epsnet}, and the union bound, we conclude that 
\begin{equation} \label{eq:ndkksum2}
	m_n(z) = \frac{1}{n} \sum_{k=1}^n \frac{1}{d_{kk} - z - m_n(z) + O\left( \frac{1}{\log^3 n} \right)} + O \left( \frac{1}{\log^3 n} \right)
\end{equation}
with (uniformly) overwhelming probability for all $z \in Q$.  

From \eqref{eq:ndkksum2}, we see that $\sup_{z \in Q} |m_n(z)| \ll 1$ with overwhelming probability.  We now take advantage of the particular form $\mu_D$ takes under the assumptions of condition {\bf C1}.   In particular, let $l \geq 1, 0 < p_1, \ldots, p_l \leq 1, a_1, \ldots, a_l \in \mathbb{R}$ be the constants from Definition \ref{def:C1}.  Then for any $1 \leq k \leq n$, there exists $1 \leq i \leq l$ such that $d_{kk} = a_i$.  

Suppose $a_i - z - m_n(z) = o(1)$ (or some subsequence $a_i - z - m_{n_r}(z) = o(1)$) for some $1 \leq i \leq l$ and some $z \in Q$.  Then it must be the case that $\Im(z) = o(1)$ since the imaginary parts of $m_n(z)$ and $z$ are both positive in the region $Q$.  We now consider the set of indices 
$$ C := \{1 \leq k \leq n : d_{kk} = a_i \}. $$
From condition {\bf C1}, it follows that 
$$ \frac{1}{n} \sum_{k \notin C} \frac{1}{d_{kk} - z - m_n(z) + O\left( \frac{1}{\log^3 n} \right)} = O(1) $$
and $|C| = \Theta(n)$.  This implies that $m_n(z)$ is unbounded, a contradiction.  Thus, by another $\eps$-net argument, we conclude that $m_n(z) - z - a_i$ is bounded away from zero for all $1 \leq i \leq l$ and $z \in Q$ with overwhelming probability.  

Therefore, from \eqref{eq:ndkksum2}, we have 
\begin{equation} \label{eq:condest1}
	m_n(z) = g_n(z + m_n(z)) + \left( \frac{1}{\log^3 n} \right)
\end{equation}
with (uniformly) overwhelming probability for all $z \in Q$.  By the assumptions of condition {\bf C1}, we have that
\begin{equation} \label{eq:condest2}
	g_n(z) = g(z) + O \left( \frac{1}{n \Im(z)} \right) 
\end{equation}
uniformly for any $z \in \C^+$.  Thus, we conclude that
$$ m_n(z) = g(z + m_n(z)) + \left( \frac{1}{\log^3 n} \right) $$
with (uniformly) overwhelming probability for all $z \in Q$.  The proof of Lemma \ref{lemma:stbound} is complete. 
\end{proof}

\begin{remark} \label{rem:C1}
In the proof of Lemma \ref{lemma:stbound}, we use the fact that $\{D_n\}_{n \geq 1}$ satisfies condition {\bf C1} in order to establish \eqref{eq:condest1} and \eqref{eq:condest2}.  In particular, \eqref{eq:condest2} allows us to replace the function $g_n$ by $g$.  This is an important step because we must replace $g_n$ by $g$ in order to apply the stability results from Section \ref{sec:stability}.  
\end{remark}

Since $\|W_n\| \leq \frac{1}{\sqrt{n}} \|M_n\| + \|D_n\| = O(1)$ with overwhelming probability\footnote{In particular, a bound for $\|M_n\|$ can be obtained by using that $\|M_n\| \leq 2 \sup_x\|M_n x\|$ where $x$ ranges over a $1/2$-net of the unit ball.  Thus, by the union bound and a concentration of measure result one can obtain that $\frac{1}{\sqrt{n}} \|M_n\| = O(1)$ with overwhelming probability.} (see for example \cite[Chapter 5]{BSbook}), we obtain the following theorem as a consequence of Lemma \ref{lemma:control} and Lemma \ref{lemma:stbound}.

\begin{theorem}[Concentration of the ESD up to the edge] \label{thm:concentration}
Let $M_n$ satisfy condition {\bf C0}.  Suppose the entries of $M_n$ are bounded by $K$ for some $K \geq 1$ where $\frac{K}{\sqrt{n}} \ll \frac{1}{\log^3 n}.$  Let $\{D_n\}_{n \geq 1}$ satisfy condition {\bf C1} with measure $\mu_D$.  Assume the limiting density $\rho$ is bounded.  Let $I$ be an interval such that
$$ |I| \geq \frac{K^2 \log^{22} n}{n}. $$
Then 
$$ N_{I} = n\int_{I} \rho(x) dx + o(n|I|) $$
with overwhelming probability.  
\end{theorem}

With Theorem \ref{thm:concentration} in hand, we now complete the proof of Theorem \ref{thm:corr} assuming Theorem \ref{thm:4moment}.  The proof is based on the arguments from \cite{TVuniv}.  

\begin{proof}[Proof of Theorem \ref{thm:corr}]
Fix $k\geq 1$, and let $x_0, W_n, W_n'$ be as in Theorem \ref{thm:corr}.  It suffices to show that 
\begin{equation} \label{eq:rhokf}
	\int_\mathbb{R} f(u_1, \ldots, u_k) \left( \frac{1}{n \rho(x_0) } \right)^k \rho_n^{(k)} \left(x_0 + \frac{u_1}{n \rho(x_0)}, \ldots, x_0 + \frac{u_k}{n \rho(x_0)} \right) du_1 \ldots du_k
\end{equation}
only changes by $o(1)$ when $\rho_n^{(k)}$ is replaced by ${\rho'}_n^{(k)}$, for any fixed test function $f$.  By an approximation argument, we can take $f$ to be smooth.  

Using \eqref{cord}, we rewrite \eqref{eq:rhokf} as
\begin{equation} \label{eq:rhokfe}
	\E \sum_{\substack{i_1, \ldots, i_k\\ \text{distinct}}} f \left( (\lambda_{i_1}(A_n) - x_0 n) \rho(x_0), \ldots, (\lambda_{i_1}(A_n) - x_0 n) \rho(x_0) \right), 
\end{equation}
where $A_n = n W_n$.  Set $A_n' = n W_n'$; by Theorem \ref{thm:4moment}, we have
\begin{align*}
	\E f &\left( (\lambda_{i_1}(A_n) - x_0 n) \rho(x_0), \ldots, (\lambda_{i_1}(A_n) - x_0 n) \rho(x_0) \right) \\
	&= \E f \left( (\lambda_{i_1}(A'_n) - x_0 n) \rho(x_0), \ldots, (\lambda_{i_1}(A'_n) - x_0 n) \rho(x_0) \right) + O(n^{-c_0})
\end{align*}
for each $i_1, \ldots, i_k$ in the $(\eps,n)$-bulk of $\rho$.  By Theorem \ref{thm:concentration}, only $O(n^c)$ eigenvalues contribute to \eqref{eq:rhokfe}, where $c>0$ can be made arbitrarily small.  Thus, the claim follows from the triangle inequality (choosing $c$ small enough compared to $c_0$).  
\end{proof}

It remains to prove Theorem \ref{thm:4moment} and Theorem \ref{thm:gap}.

\section{Delocalization of the Eigenvectors} \label{sec:eigenvectors}

We now use Theorem \ref{thm:concentration} to study the eigenvectors of $W_n$.  In particular, this section is devoted to the following theorem.  

\begin{theorem}[Delocalization of the eigenvectors] \label{thm:eigenvectors}
Let $M_n$ be a $n \times n$ matrix that satisfies condition {\bf C0}.  Suppose the entries of $M_n$ are bounded by $K \geq 1$, where $K \ll \log^{O(1)} n$.  Let $\{D_n\}_{n \geq 1}$ satisfy condition {\bf C1} with measure $\mu_D$.  Assume the limiting density $\rho$ satisfies condition {\bf (A)}.  Let $\eps > 0$ be a small parameter.  If $i$ is in the $(\eps,n)$-bulk of $\rho$, then
$$ \sup_{1 \leq j \leq n} |u_i(W_n)^\ast e_j| \ll_\eps \frac{\log^{O(1)} n}{\sqrt{n}} $$
with overwhelming probability, where $e_1, \ldots, e_n$ is the standard basis.  
\end{theorem}

We remind the reader that $s_0(\rho), s_1(\rho), \ldots, s_q(\rho)$ are the specific quantiles of $\rho$ defined in \eqref{eq:def:svalues}.  We begin with the following lemmas.

\begin{lemma} \label{lemma:epsbulk}
Let $M_n$ satisfy condition {\bf C0} and $\{D_n\}_{n \geq 1}$ satisfy condition {\bf C1} with measure $\mu_D$.  Suppose the entries of $M_n$ are bounded by $K$, where $K$ satisfies the assumptions of Theorem \ref{thm:concentration}.  Assume the density $\rho$ satisfies condition {\bf (A)}.  In particular, let $[\alpha_j, \beta_j], j = 1,\ldots,q$ be the intervals from Definition \ref{def:A}.  Let $\eps > 0$ be a small parameter.  Then there exists $\delta, c_\eps > 0$ such that the following holds.
\begin{enumerate}[(a)]
\item The intervals $(\alpha_j + \delta, \beta_j - \delta), j=1,\ldots,q$ are non-empty.
\item $\rho(x) \geq c_\eps$ for all $x \in \cup_{j=1}^q [\alpha_j +\delta, \beta_j - \delta]$. \label{bulk:non0}
\item If the index $i$ is in the $(\eps,n)$-bulk of $\rho$, then $\lambda_i(W_n) \in \cup_{j=1}^q (\alpha_j +\delta, \beta_j - \delta)$ with overwhelming probability.  In particular, if $(s_{j-1}(\rho) + \eps)n \leq i \leq (s_j(\rho) - \eps)n$ for some $j=1,\ldots,q$, then $\lambda_i(W_n) \in (\alpha_j + \delta, \beta_j -\delta)$ with overwhelming probability.  \label{bulk:eps}
\end{enumerate}
\end{lemma}

\begin{proof}
Define $\delta>0$ such that
$$ \int_{\alpha_j}^{\alpha_j + \delta} \rho(x) dx \leq \frac{\eps}{2} \quad \text{and} \quad \int_{\beta_{j} - \delta}^{\beta_j} \rho(x) dx \leq \frac{\eps}{2} \quad \text{for} \quad j=1,\ldots,q.  $$
Since we can always take $\delta$ smaller and still satisfy the above conditions, we assume $\delta$ is sufficiently small such that the intervals $(\alpha_j + \delta, \beta_j - \delta), j=1,\ldots,q$ are non-empty.

Since $\rho$ is a continuous probability density, it follows that $\rho$ attains its minimum on the compact set $\cup_{j=1}^q [\alpha_j +\delta, \beta_j - \delta]$.  Assumption \eqref{eq:condAnon0} from condition {\bf (A)} ensures that the minimum is non-zero.  This proves part \eqref{bulk:non0}.  

We now prove part \eqref{bulk:eps}.  Since $\|D_n\| = O(1)$, it follows that there exists $C>0$ such that $\lambda_i(W_n) \in [-C,C]$ for all $1 \leq i \leq n$ with overwhelming probability (see \cite[Chapter 5]{BSbook} or the previous remarks regarding $\frac{1}{\sqrt{n}} \|M_n\|$).   By Theorem \ref{thm:concentration}, with overwhelming probability,
$$ N_I = o(n) $$
where $I = [-C,C] \setminus \cup_{j=1}^q [\alpha_j, \beta_j]$.  Suppose $i$ is in the $(\eps,n)$-bulk of $\rho$.  Then there exists $1 \leq j \leq q$ such that
$$ (s_{j-1}(\rho) + \eps)n \leq i \leq (s_{j}(\rho) - \eps)n. $$
We will show that, with overwhelming probability, $\lambda_i(W_n) \in (\alpha_j + \delta, \beta_j - \delta)$.  By definition of $\delta$ above and Theorem \ref{thm:concentration}, we have that
\begin{align*}
	N_{(-\infty, \alpha_j + \delta]} = N_{[-C,\alpha_j + \delta]} &= n s_{j-1} + o(n) + N_{[\alpha_j,\alpha_j + \delta]} \\
	&\leq n s_{j-1}(\rho) + \frac{\eps}{2} n + o(n) \\
	&\leq n s_{j-1}(\rho) + \frac{3}{4} \eps n 
\end{align*}
with overwhelming probability.  Thus $\lambda_i(W_n) > \alpha_j + \delta$ with overwhelming probability.  Similarly, since 
$$ N_{[\beta_j - \delta, \beta_j]} \leq \frac{\eps}{2} n + o(n) \leq \frac{3}{4}\eps n $$
with overwhelming probability, it follows that 
$$ N_{[-C, \beta_j - \delta]} \geq n s_j(\rho) - \frac{3}{4} \eps n. $$
Therefore, we conclude that $\alpha_j + \delta < \lambda_i(W_n) < \beta_{j} - \delta$ with overwhelming probability.  
\end{proof}

From \cite{TVuniv}, we have the following result.

\begin{lemma} \label{lemma:eigenvector_entry}
Let 
$$ A_n = \begin{pmatrix} a & X^\ast \\ X & A_{n-1} \end{pmatrix} $$
be a $n \times n$ Hermitian matrix for some $a \in \R$ and $X \in \C^{n-1}$, and let $\begin{pmatrix} x \\ v\end{pmatrix}$ be a unit eigenvector of $A_n$ with eigenvalues $\lambda_i(A_n)$, where $x \in \C$ and $v \in \C^{n-1}$.  Suppose that none of the eigenvalues of $A_{n-1}$ are equal to $\lambda_i(A_n)$.  Then
$$ |x|^2 = \frac{1}{ 1 + \sum_{j=1}^{n-1} (\lambda_j(A_{n-1})-\lambda_i(A_{n}))^{-2} |u_j(A_{n-1})^\ast X|^2 } $$
where $u_j(A_{n-1})$ is a unit eigenvector corresponding to the eigenvalue $\lambda_j(A_{n-1})$.  
\end{lemma}

We now prove Theorem \ref{thm:eigenvectors}.  Let $\eps > 0$ and assume $i$ is in the $(\eps,n)$-bulk of $\rho$.  Thus there exists $1 \leq k \leq q$ such that $(s_{k-1}(\rho) +\eps) n \leq i \leq (s_k(\rho)-\eps)n$.  By the symmetry of $W_n$ and the union bound, it suffices to show that
$$ |u_i(W_n)^\ast e_1| \ll_{\eps} n^{-1/2} \log^{O(1)} n $$
with overwhelming probability.  

The continuity of the entries of $M_n$ ensure that the hypothesis of Lemma \ref{lemma:eigenvector_entry} is obeyed almost surely.  Using Lemma \ref{lemma:eigenvector_entry} and the arguments in \cite{TVedge}, our task reduces to verifying that
\begin{equation} \label{eq:showsumjt}
	\sum_{j: i-T_- \leq j \leq i+T_+} \frac{ |u_j(W_{n-1})^\ast Y|^2}{|\lambda_j(W_{n-1}) - \lambda_i(W_n)|} \gg_{\eps} \log^{-O(1)} n 
\end{equation}
with overwhelming probability for some $1 \leq T_-,T_+ \ll \log^{O(1)} n$, where $Y$ is the bottom left $(n-1) \times 1$ vector of $\frac{1}{\sqrt{n}} M_n$ and $W_{n-1}$ is the bottom right $(n-1) \times (n-1)$ minor of $W_n$.  

By Lemma \ref{lemma:epsbulk}, there exists $\delta>0$ such that $\lambda_i(W_n) \in (\alpha_k+\delta, \beta_k-\delta)$ with overwhelming probability.  Moreover, there exists $c_\eps>0$ such that $\rho(x) \geq c_\eps$ for all $x \in (\alpha_k+\delta, \beta_k-\delta)$.  

Let $A$ be a large constant to be chosen later.  By partitioning the interval $(\alpha_k+\delta, \beta_k-\delta)$, it follows that there exists an interval $I$ of length $\log^A n/n$ that contains $\lambda_i(W_n)$ with overwhelming probability.  By Theorem \ref{thm:concentration}, $I$ contains $\Theta_{\eps}(\log^A n)$ eigenvalues of $W_n$ with overwhelming probability.  We now note that if $\lambda_j(W_n), \lambda_{j+1}(W_n) \in I$, then by Cauchy's interlacing property, 
$$ |\lambda_{j}(W_{n-1}) - \lambda_{i}(W_n)| \ll \frac{\log^A n}{n}. $$
Thus, for a suitable choice of $T_-,T_+$, 
\begin{align*}
	\sum_{j: i-T_- \leq j \leq i+T_+} \frac{ |u_j(W_{n-1})^\ast Y|^2}{|\lambda_j(W_{n-1}) - \lambda_i(W_n)|} & \gg \sum_{j: \lambda_j(W_n),\lambda_{j+1}(W_n) \in I} \frac{ |u_j(W_{n-1})^\ast Y|^2}{|\lambda_j(W_{n-1}) - \lambda_i(W_n)|} \\
	& \gg n \log^{-A} n \sum_{j: \lambda_j(W_n),\lambda_{j+1}(W_n) \in I} |u_j(W_{n-1})^\ast Y|^2 \\
	& \gg \log^{-A} n \| \pi_H(X) \|^2
\end{align*}
where $H$ is the span of $u_j(W_{n-1})$ for $\lambda_j(W_n), \lambda_{j+1}(W_n) \in I$ and $X = \sqrt{n} Y$.  Condition \eqref{eq:showsumjt} now follows from Lemma \ref{lemma:proj} by taking $A$ sufficiently large.  The proof of Theorem \ref{thm:eigenvectors} is complete.  

\section{Lower Bound on the Eigenvalue Gap} \label{sec:gap}

This section is devoted to Theorem \ref{thm:gap}.  Define 
$$ A_n := \sqrt{n} M_n + n D_n = n W_n. $$  

We will need the following deterministic lemma from \cite{ESY3} (see also \cite{TVuniv}). 

\begin{lemma}[Interlacing identity]
Let $A_n$ be an $n \times n$ Hermitian matrix, let $A_{n-1}$ be the top $n-1 \times n-1$ minor, let $a_{nn}$ be the bottom right component, and let $X \in \C^{n-1}$ be the rightmost column with the bottom entry $a_{nn}$ removed.  Suppose that $X$ is not orthogonal to any of the unit eigenvectors $u_j(A_{n-1})$ of $A_{n-1}$.  Then we have
\begin{equation} \label{eq:interlace}
	\sum_{j=1}^{n-1} \frac{ |u_j(A_{n-1})^\ast X|^2 }{\lambda_j(A_{n-1}) - \lambda_{i}(A_n) } = a_{nn} - \lambda_{i}(A_n)
\end{equation}
for every $1 \leq i \leq n$.
\end{lemma}

We now prove Theorem \ref{thm:gap}.  Fix $\eps, c_0>0$.  We write $n_0, i_0$ for $n,i$; thus $(s_{k-1}(\rho) + \eps)n_0 \leq i_0 \leq (s_k(\rho) - \eps) n_0$ for some $k=1,\ldots,q$.  It thus suffices to show that
$$ \lambda_{i_0+1}(A_{n_0}) - \lambda_{i_0}(A_{n_0}) > n_0^{-c_0} $$
with high probability.    

For each $n_0/2 \leq n \leq n_0$ let $A_n$ be the top left $n \times n$ minor of $A_{n_0}$.  Following \cite{TVuniv,TVedge} we define the regularized gap
$$ g_{i,l,n} := \inf_{1 \leq i_- \leq i-l < i_- \leq i_+ \leq n} \frac{\lambda_{i_+}(A_n) - \lambda_{i_-}(A_n)}{\min(i_+-i_-,\log^{C_1} n_{0})^{\log^{0.9} n_0}}, $$
for all $n_0/2 \leq n \leq n_0$ and $1 \leq i-l < i \leq n$, where $C_1$ is a large constant to be chosen later.  It will suffice to show that 
$$ g_{i_0,1,n_0} > n^{-c_0} $$
with high probability.  Let $X_n$ be the rightmost column of $A_{n+1}$ with the bottom coordinate removed.  We then have the following deterministic lemma, which is a slightly modified version of \cite[Lemma 51]{TVuniv}.

\begin{lemma}[Backwards propagation of gap] \label{lemma:propagation}
Suppose that $n_0/2 \leq n < n_0$ and $l \leq \eps n/10$ is such that
$$ g_{i_0,l,n+1} \leq \delta $$
for some $0 < \delta \leq 1$ (which can depend on $n$), and that
$$ g_{i_0,l+1,n} \geq 2^m g_{i_0,l,n+1} $$
for some $m \geq 0$ with
$$ 2^m \leq \delta^{-1/2}. $$
Then one of the following statements hold:
\begin{enumerate}[(i)]
\item (Macroscopic spectral concentration) \label{item:macro} There exists $1 \leq i_- < i_+ \leq n+1$ with $i_+ - i_- \geq \log^{C_1/2} n$ such that
$$ |\lambda_{i_+}(A_{n+1}) - \lambda_{i_-}(A_{n+1})| \leq \delta^{1/4} \exp(\log^{0.95} n)(i_+ - i_i). $$
\item (Small inner product) \label{item:smallip} There exists $\eps n/2 \leq i_- \leq i_0 - l < i_0 \leq i_+ \leq (1-\eps/2)n$ with $i_+-i_- \leq \log^{C_1/2} n$ such that
$$ \sum_{i_- \leq j < i_+} |u_j(A_n)^\ast X_n|^2 \leq \frac{n(i_+ - i_-)}{2^{m/2} \log^{0.01} n}. $$
\item (Large coefficient) \label{item:coef} We have
$$ |a_{n+1,n+1}| \geq \frac{ n \exp(-\log^{0.95} n) }{ 2 \delta^{1/2} }. $$
\item (Large eigenvalue) \label{item:eigenvalue} For some $1 \leq i \leq n+1$ one has
$$ |\lambda_i(A_{n+1})| \geq \frac{ n \exp(-\log^{0.95} n) }{2 \delta^{1/2}}. $$
\item (Large inner product in bulk) \label{item:bulkip} There exists $\eps n/10 \leq i \leq (1-\eps/10)n$ such that
$$ |u_i(A_n)^\ast X_n|^2 \geq \frac{n \exp(-\log^{0.96} n)}{\delta^{1/2}}. $$
\item (Large row) \label{item:row} We have
$$ \|X_n\|^2 \geq \frac{n^2 \exp(-\log^{0.96} n) }{\delta^{1/2}}. $$
\item (Large inner product near $i_0$) \label{item:largeip} There exits $\eps n/10 \leq i \leq (1-\eps/10) n$ with $|i-i_0| \leq \log^{C_1} n$ such that
$$ |u_i(A_n)^\ast X_n|^2 \geq 2^{m/2} n \log^{0.8} n. $$
\end{enumerate}
\end{lemma}
\begin{remark}
Lemma \ref{lemma:propagation} is almost identical to \cite[Lemma 51]{TVuniv} except for conditions \eqref{item:coef} and \eqref{item:eigenvalue} which have only been modified slightly.  
\end{remark}
\begin{proof}[Proof of Lemma \ref{lemma:propagation}]
The proof of Lemma \ref{lemma:propagation} follows the proof of \cite[Lemma 51]{TVuniv} in \cite[Section 6]{TVuniv} almost exactly.  Only the following changes have to be made.
\begin{itemize}
\item Under the assumption that conditions \eqref{item:coef} and \eqref{item:eigenvalue} fail, 
$$ |a_{n+1,n+1} - \lambda_{i}(A_{n+1})| \leq \frac{n \exp(-\log^{0.95} n)}{\delta^{1/2}} $$
by the triangle inequality.  Since this is the same bound obtained in \cite[Section 6]{TVuniv}, no further changes are required. 
\end{itemize}
\end{proof}

Next, we show that each of the bad events \eqref{item:macro}-\eqref{item:largeip} occur with small probability.  

\begin{proposition}[Bad events are rare] \label{prop:bad}
Suppose that $n_0/2 \leq n < n_0$ and $l \leq \eps n/10$, and set $\delta := n_0^{-\kappa}$ for some sufficiently small fixed $\kappa >0$.  Then:
\begin{enumerate}[(a)]
\item \label{item:highprob} The events \eqref{item:macro}, \eqref{item:coef}, \eqref{item:eigenvalue}, \eqref{item:bulkip}, and \eqref{item:row} in Lemma \ref{lemma:propagation} all fail with high probability.
\item \label{item:condition} There is a constant $C'$ such that all the coefficients of the eigenvectors $u_j(A_n)$ for $j$ in the $(\eps/2,n)$-bulk of $\rho$ are of magnitude at most $n^{-1/2} \log^{C'} n$ with overwhelming probability.  Conditioning $A_n$ to be a matrix with this property, the events \eqref{item:smallip} and \eqref{item:largeip} occur with a conditional probability of at most $2^{-\kappa m} + n^{-\kappa}$.  
\item \label{item:condition2} Furthermore, there is a constant $C_2$ (depending on $C',\kappa,C_1$) such that if $l \geq C_2$ and $A_n$ is conditioned as in \eqref{item:condition}, then \eqref{item:smallip} and \eqref{item:largeip} in fact occur with a conditional probability of at most $2^{-\kappa m}\log^{-2C_1} n + n^{-\kappa}$.  
\end{enumerate}
\end{proposition}
\begin{proof}
Event \eqref{item:macro} fails with overwhelming probability due to Theorem \ref{thm:concentration}.  For $n$ sufficiently large, the events \eqref{item:coef} and \eqref{item:row} are empty due to the truncation assumption.  Moreover, event \eqref{item:bulkip} fails with overwhelming probability due to the truncation assumption and Lemma \ref{lemma:proj} for $n$ large enough.  Since the operator norm of $A_n$ is $O(n)$ with overwhelming probability (see \cite[Chapter 5]{BSbook} or our previous remarks regarding the norm of $W_n$), it follows that event \eqref{item:eigenvalue} fails with overwhelming probability.  The proof of \eqref{item:highprob} is complete.

We now prove \eqref{item:condition} and \eqref{item:condition2} at the same time.  By Theorem \ref{thm:eigenvectors}, there exists a positive constant $C'$ such that all the coefficients of the eigenvectors $u_j(A_n)$ for $j$ in the $(\eps/2,n)$-bulk of $\rho$ are of magnitude at most $n^{-1/2} \log^{C'} n$ with overwhelming probability.  

We begin with event \eqref{item:largeip}.  There are two cases to consider.  If $2^m \geq \log^{C_3} n$ for some sufficiently large constant $C_3$, then event \eqref{item:largeip} fails with overwhelming probability due to Lemma \ref{lemma:proj} and the truncation assumption.  If $2^m \leq \log^{O(1)} n$, we need to show that 
\begin{equation} \label{eq:psi}
	\P(|S_i| \geq 2^{m/4} \log^{0.4} n) \leq 2^{-\kappa m} \log^{-2C_1} n + n^{-\kappa}, 
\end{equation}
where $S_i \in \C$ is the random walk
\begin{equation} \label{eq:sidef}
	S_i := \zeta_{1,n+1,} w_{i,1} + \cdots + \zeta_{n,n+1} w_{i,n} 
\end{equation}
and $w_{i,1}, \ldots, w_{i,n}$ are the coefficients of the unit eigenvector $u_i(A_n)$.  Since $S_i$ has mean zero and unit variance, Lemma \ref{lemma:tail} gives the desired bound in \eqref{eq:psi} by taking $t=2^{m/4} \log^{0.4} n$ and $\sigma = n^{-1/2} \log^{O(1)} n$.  

We now consider the event \eqref{item:smallip}.  Again we consider two cases.  If $i_+-i_- \geq \log^{C_3} n$ for some sufficiently large constant $C_3$, then \eqref{item:smallip} fails with overwhelming probability by Lemma \ref{lemma:proj} and the truncation assumption.  Assume $i_+ - i_- \leq \log^{O(1)} n$.  By Theorem \ref{thm:concentration} and condition {\bf (A)}, this implies that $i_+$ and $i_-$ are in the $(\eps/2,n)$-bulk of $\rho$ with overwhelming probability.  We can rewrite event \eqref{item:smallip} as
$$ \|\vec{S}\| \leq \frac{(i_+ - i_-)^{1/2}}{2^{m/4} \log^{0.005} n} $$
where $\vec{S} \in \C^{i_+-i_-}$ with components $S_j$ defined as in \eqref{eq:sidef}.  Since the above event is non-increasing in $m$, without loss of generality we assume $2^m \leq n^{1/100}$ (say).  So by Lemma \ref{lemma:tail}, we conclude 
\begin{equation} \label{eq:pvecst}
	\P(\|\vec{S}\|\leq t) \ll \left( \frac{t}{(i_+ - i_-)^{1/2}} \right)^{(i_+ - i_-)/4} + n^{-1/2} t^{-3} \log^{O(1)} n. 
\end{equation}
Taking $t = \frac{(i_+ - i_-)^{1/2}}{2^{m/4} \log^{0.005} n}$ and the trivial bound $i_+ - i_i \geq 1$ completes the proof of \eqref{item:condition}.  Applying \eqref{eq:pvecst} with $t = \frac{(i_+ - i_-)^{1/2}}{2^{m/4} \log^{0.005} n}$ and $i_+ - i_- \geq l \geq C_2$ yields 
$$ \P \left( \|\vec{S}\| \leq \frac{(i_+ - i_-)^{1/2}}{2^{m/4} \log^{0.005} n} \right) \ll (2^{m/4} \log^{0.005} n)^{-C_2/4} + n^{-1/2} 2^{3m/4} \log^{O(1)} n. $$
This completes the proof of \eqref{item:condition2} by taking $C_2$ sufficiently large (and recalling that $2^m \leq n^{1/100}$).  
\end{proof}

Using Lemma \ref{lemma:propagation} and Proposition \ref{prop:bad}, the proof of Theorem \ref{thm:gap} can be concluded by following the proof of \cite[Theorem 19]{TVuniv} in \cite[Section 3.5]{TVuniv}.  Only the following changes have to be made:
\begin{itemize}
\item In the definition of the event $E_n$, the range $\eps n/2 \leq j \leq (1-\eps/2)n$ needs to be changed to include only $j$ in the $(\eps/2,n)$-bulk of $\rho$.
\item All references to \cite[Lemma 51]{TVuniv} need to be replaced with Lemma \ref{lemma:propagation}.
\item All references to \cite[Proposition 53]{TVuniv} need to be replaced with Proposition \ref{prop:bad}.
\end{itemize}

\section{The Four Moment Theorem} \label{sec:4moment}

This section is devoted to Theorem \ref{thm:4moment}.  We begin with the following result from \cite{TVuniv}.

\begin{proposition}[Replacements given a good configuration] \label{prop:replacement}
There exists a positive constant $C_1$ such that the following holds.  Let $k \geq 1$ and $\eps_1>0$, and assume $n$ sufficiently large depending on these parameters.  Let $1 \leq i_1 < \cdots < i_k \leq n$.  For a complex parameter $z$, let $A(z)$ be a (deterministic) family of $n \times n$ Hermitian matrices of the form
$$ A(z) = A(0) + z e_p e^\ast_q + \bar{z} e_q e^\ast_p $$
where $e_p, e_q$ are unit vectors.  We assume that for every $1 \leq j \leq k$ and every $|z| \leq n^{1/2+\eps_1}$ whose real and imaginary parts are multiples of $n^{-C_1}$, we have
\begin{itemize}
\item (Eigenvalue separation) For any $1 \leq i \leq n$ with $|i-i_j| \geq n^{\eps_1}$, we have
\begin{equation} \label{eq:evaluesep}
	|\lambda_i(A(z)) - \lambda_{i_j}(A(z))| \geq n^{-\eps_1}|i-i_j|. 
\end{equation}
\item (Delocalization at $i_j$) If $P_{i_j}(A(z))$ is the orthogonal projection to the eigenspace associated to $\lambda_{i_j}(A(z))$, then
\begin{equation} \label{eq:delocal}
\| P_{i_j}(A(z))e_p \|, \|P_{i_j}(A(z))e_q \| \leq n^{-1/2 + \eps_1}. 
\end{equation}
\item For every $\alpha \geq 0$, 
\begin{equation} \label{eq:delocal2}
	\| P_{i_j,\alpha}(A(z))e_p\|,\|P_{i_j,\alpha}(A(z)) e_q \| \leq 2^{\alpha/2} n^{-1/2+\eps_1} 
\end{equation}
whenever $P_{i_j,\alpha}$ is the orthogonal projection to the eigenspaces corresponding to eigenvalues $\lambda_i(A(z))$ with $2^\alpha \leq |i-i_j| < 2^{\alpha+1}$.
\end{itemize}
We say that $A(0)$,$e_p$, $e_q$ are a good configuration for $i_1,\ldots, i_k$ if the above properties hold.  Assuming this good configuration, then we have
\begin{equation} \label{eq:efzefz'}
	\E F(\zeta) = \E F(\zeta') + O(n^{-(r+1)/2+O(\eps_1)})
\end{equation}
whenever
$$ F(z) := G(\lambda_{i_1}(A(z)), \ldots, \lambda_{i_k}(A(z)), Q_{i_1}(A(z)), \ldots, Q_{i_k}(A(z))), $$
and 
$$ G = G(\lambda_{i_1}, \ldots, \lambda_{i_k}, Q_{i_1}, \ldots, Q_{i_k}) $$
is a smooth function from $\R^k \times \R^k_+ \rightarrow \R$ that is supported on the region
$$ Q_{i_1}, \ldots, Q_{i_k} \leq n^{\eps_1} $$
and obeys the derivative bounds
$$ |\nabla^jG| \leq n^{\eps_1} $$
for all $0 \leq j \leq 5$, and $\zeta,\zeta'$ are random variables with $|\zeta|,|\zeta'| \leq n^{1/2+\eps_1}$ almost surely, which match to order $r$ for some $r=2,3,4$.  

If $G$ obeys the improved derivative bounds
$$ |\nabla^jG| \leq n^{-Cj\eps_1} $$
for $0 \leq j \leq 5$ and some sufficiently large absolute constant $C$, then we can strengthen $n^{-(r+1)/2+O(\eps_1)}$ in \eqref{eq:efzefz'} to $n^{-(r+1)/2-\eps_1}$. 
\end{proposition}

We will verify that the good configuration property holds with overwhelming probability.  

\begin{proposition}[Good configurations occur very frequently] \label{prop:config}
Let $\eps, \eps_1 > 0$ be small parameters, and $C,C_1,k\geq 1$.  Let $M_n = (\zeta_{ij})_{1 \leq i,j \leq n}$ satisfy condition {\bf C0}.  Let $\{D_n\}_{n \geq 1}$ satisfy condition {\bf C1}.  Assume the density $\rho$ satisfies condition {\bf (A)} and let $ i_1 < \cdots < i_k$ be in the $(\eps,n)$-bulk of $\rho$.  Let $1 \leq p,q \leq n$ and let $e_1, \ldots, e_n$ be the standard basis of $\C^n$.  Let $A(0) = \sqrt{n} M_n(0) + n D_n$ where $M_n(0)$ is the matrix $M_n$ with $\zeta_{pq} = \zeta_{qp} = 0$ and $|\zeta_{ij}| \leq \log^{C} n$ for all $1 \leq i,j \leq n$.  Then $A(0), e_p, e_q$ obey the Good configuration Condition in Proposition \ref{prop:replacement} for $i_1, \ldots, i_k$ and with the indicated value of $\eps_1,C_1$ with overwhelming probability.
\end{proposition}
\begin{proof}
By the union bound it suffices to fix $1 \leq j \leq k$ and $|z| \leq n^{1/2+\eps_1}$ whose real and imaginary parts are multiples of $n^{-C_1}$.  By a further application of the union bound and Theorem \ref{thm:concentration}, condition \eqref{eq:evaluesep} holds with overwhelming probability for every $1 \leq i \leq n$ with $|i - i_j| \geq n^{\eps_1}$.  Condition \eqref{eq:delocal} holds with overwhelming probability by Theorem \ref{thm:eigenvectors}.  A similar argument using Pythagoras' theorem gives \eqref{eq:delocal2} with overwhelming probability, unless the eigenvalues $\lambda_i(A(z))$ contributing to \eqref{eq:delocal2} are indexed by $i$ not contained in the $(\eps/2,n)$-bulk of $\rho$.  By Lemma \ref{lemma:epsbulk} a contribution from outside the bulk would imply that $2^\alpha \gg_\eps n$, in which case \eqref{eq:delocal2} follows from the crude bound $\|P_{i_j, \alpha}(A(z))e_p \|, \|P_{i_j,\alpha}(A(z))e_q\| \leq 1$.   
\end{proof}

Using Proposition \ref{prop:config}, the proof of Theorem \ref{thm:4moment} repeats the proof of \cite[Theorem 15]{TVuniv} in \cite[Section 3]{TVuniv} almost exactly.  Only the following changes have to be made:
\begin{itemize}
\item All references to \cite[Theorem 19]{TVuniv} need to be replaced with Theorem \ref{thm:gap}.
\item All references to \cite[Proposition 46]{TVuniv} need to be replaced with Proposition \ref{prop:replacement}.
\item All references to \cite[Proposition 48]{TVuniv} need to be replaced with Proposition \ref{prop:config}.
\item The reference to the bulk in the proof of \cite[Lemma 49]{TVuniv} needs to be replaced with the $(\eps,n)$-bulk of $\rho$.
\end{itemize}

\end{document}